\documentclass[a4paper,11pt]{amsart}
\usepackage{amssymb}
\usepackage{amscd}
\usepackage{comment}
\usepackage{amsmath,amsthm}
\usepackage[colorlinks=true]{hyperref}
\usepackage{enumerate}
\usepackage{booktabs,multirow}
\usepackage{tikz}
\usepackage{rotating}
\usetikzlibrary{patterns}
\usetikzlibrary{decorations.pathreplacing}
\usetikzlibrary{calc,through}

\allowdisplaybreaks[1]
\numberwithin{figure}{section}
\numberwithin{equation}{section}

\title{Heaps of pieces for lattice paths}
\author[K.~Shigechi]{Keiichi~Shigechi}
\email{k1.shigechi AT gmail.com}
\date{\today}

\newcommand\tikzpic[2]{
\raisebox{#1\totalheight}{
\begin{tikzpicture}
#2
\end{tikzpicture}
}}
\newtheorem{example}[figure]{Example}
\newtheorem{lemma}[figure]{Lemma}
\newtheorem{defn}[figure]{Definition}
\newtheorem{prop}[figure]{Proposition}
\newtheorem{cor}[figure]{Corollary}

\newtheorem{remark}[figure]{Remark}
\begin{document}

\begin{abstract}
We study heaps of pieces for lattice paths, which give a combinatorial visualization of 
lattice paths. We introduce two types of heaps: type $I$ and type $II$.
A heap of type $I$ is characterized by peaks of a lattice path. 
We have a duality between a lattice path $\mu$ and its dual $\overline{\mu}$ on heaps of type $I$. 
A heap of type $II$ for $\mu$ is characterized by the skew shape between the lowest path and $\mu$.
We give a determinant expression for the generating function of heaps for general 
lattice paths, and an explicit formula for rational $(1,k)$-Dyck paths by using the inversion lemma.
We introduce and study heaps in $k+1$-dimensions which are bijective to heaps of type $II$
for $(1,k)$-Dyck paths.
Further, we show a bijective correspondence between type $I$ and type $II$ in the case of rational 
$(1,k)$-Dyck paths.
As another application of heaps, we give two explicit formulae for the generating function of heaps
for symmetric Dyck paths in terms of statistics on Dyck paths and on symmetric Dyck paths respectively.
\end{abstract}

\maketitle

\section{Introduction}
The notion of heaps of pieces was introduced by X.~G.~Viennot \cite{Vie86} to give
a geometric visualization of the Cartier--Foata's commutation monoid.
This notion has vast applications to other research fields such as 
parallelism in computer science, directed animals in statistical physics,
and parallelogram polyominoes, Motzkin paths and orthogonal polynomials,  
Rogers--Ramanujan identities, fully commutative elements in Coxter groups in combinatorics 
(see \cite{BetPen01,BouMelRec02,BouMelVie92,Gar20,Kra06,Lal95,Ste96,Vie86,Vie87,Vie06,Vie17} and references therein).

In this paper, we study heaps of pieces for lattice paths.
Suppose that a lattice path $\mu_{0}$ is a path which consists of up steps $U$ and right steps $D$.
A heap of pieces is a visualization of a lattice path below $\mu_{0}$ and above $\mu_{1}:=D^{r}U^{s}$ 
where $r$ (resp. $s$) is the number of up (resp. right) steps in $\mu_{0}$.
We consider a class of lattice paths called rational Dyck paths.
A rational Dyck path, denoted by $(a,b)$-Dyck path, is a lattice 
path from $(0,0)$ to $(bn,an)$ and below the path $y=ax/b$ and above $y=0$.
Especially, we consider a special class of rational Dyck paths: $(1,b)$-Dyck paths and $(a,1)$-Dyck paths.
Given a lattice path $\mu_{0}$, we introduce two types of heaps of pieces, which we 
call type $I$ and type $II$.
The heaps of type $I$ is characterized by peaks of a lattice path $\mu$, a pattern $UD$ in $\mu$, 
and the heaps of type $II$ is characterized by the skew shape $\mu/\mu_{0}$.
The heaps of type $I$ and type $II$ are bijective to each other.

In type $I$, we have a duality between a lattice path $\mu$ and its dual path $\overline{\mu}$.
In the case of generalized Dyck paths, this duality is the duality between 
$(1,b)$-Dyck paths and $(b,1)$-Dyck paths.
A heap for a $(1,b)$-Dyck path consists of pieces with different sizes. 
We have constraints on the size of a piece: it depends only on its position.
Similarly, a heap for a $(b,1)$-Dyck path consists of only pieces of size one,
called a dimer.  We have some constraints on the positions of pieces.
The duality can be seen as an exchange of the roles of a size and a position 
in the constraints on heaps for $(1,b)$- and $(b,1)$-Dyck paths.
By introducing a valuation on a heap, one can define a generating function 
of heaps. We observe that there exists a duality between the generating functions 
for $\mu$ and $\overline{\mu}$.

The heaps of type $II$ are useful to compute the generating function. 
For a general lattice path, heaps give an interpretation of the generating function
in terms of non-intersecting lattice paths.
By the Lindstr\"om--Gessel--Viennot lemma, we show that the generating function 
can be expressed as a determinant.
In the case of $(1,b)$-Dyck paths, one can compute the generating functions 
by use of the extension of the inversion lemma in \cite{Kra06} (see also \cite{Vie86}), 
which connects the generating function of heaps with that of trivial heaps.
We give an interpretation of the generating function in terms of $(1,b)$-Dyck 
paths, and a functional equation for the generating function.
We also give another realization of heaps for $(1,b)$-Dyck paths.
By decomposing a $(1,b)$-Dyck path $\mathcal{D}$ of size $n$ into a $b$-tuple of 
Dyck paths of size $n$, we show a bijection between $\mathcal{D}$ and a heap 
in $b+1$-dimensions with some constraints.
The key idea is to construct an $n+1$-tuple of integer sequences of size $b$ from 
$\mathcal{D}$, and  a planar rooted tree on this $n+1$-tuple.
Then, we show that this planar rooted tree is bijective to a heap in $b+1$-dimensions, 
and the generating function of heaps in $b+1$-dimensions is expressed by the generating 
function of heaps of type $II$.

As an application of heaps, we introduce heaps for symmetric Dyck paths.
In this model, heaps consists of monomers and dimers.
The heaps satisfy similar properties to those of type $II$.
By applying the extended inversion lemma, we give a formula for the 
generating function of heaps, its expression by a continued fraction, and 
its functional equation.
We give two expressions of generating function of heaps in terms 
of statistics on symmetric Dyck paths and on Dyck paths respectively.

The paper is organized as follows.
In Section \ref{sec:HP}, we summarize the basic facts about heaps and 
introduce the extension of the inversion lemma.
In Section \ref{sec:HGDP}, we study heaps of pieces of type $I$ for generalized 
Dyck paths and show the duality between $(1,b)$-Dyck paths and $(b,1)$-Dyck paths.
In Section \ref{sec:HPLP}, the results in Section \ref{sec:HGDP} are generalized 
to the case of general lattice paths.
The heaps of type $II$ for general lattice paths are introduced in Section \ref{sec:LPGF}.
We give a determinant expression of the generating function of heaps.
In the case of $(1,k)$-Dyck paths, we study heaps in $k+1$-dimensions which are shown to be 
bijective to the heaps of type $II$. 
We construct a bijection between heaps of type $I$ and those of type $II$ in Section \ref{sec:ItoII}.
In Section \ref{sec:symDP}, we study heaps for the symmetric Dyck paths and give two formulae 
for the generating function in terms of symmetric Dyck paths and Dyck paths respectively.

\section{Heaps of pieces}
\label{sec:HP}
We introduce the notion of heaps of pieces following \cite{Vie86}.
Let $P$ be a set with a symmetric and reflexive binary relation $\mathcal{R}$.
This means that $a\mathcal{R}b\Leftrightarrow b\mathcal{R}a$ and $a\mathcal{R}a$
for every $a,b\in P$.
The relation $\mathcal{R}$ is called the {\it concurrency relation}.

\begin{defn}[{\cite[Definition 2.1]{Vie86}}]
Let $(E,\le)$ be a finite partially ordered set $E$ with the order relation $\le$.  
A labeled heap with pieces in $E$ is a triplet $(E,\le,\epsilon)$
with the map $\epsilon:E\rightarrow P$.
The map $\epsilon$ and $\le$ satisfy the following two conditions: 
\begin{enumerate}[({A}1)]
\item for every $\alpha,\beta\in E$ such that $\epsilon(\alpha)R\epsilon(\beta)$, then 
$\alpha$ and $\beta$ satisfies $\alpha\le\beta$ or $\beta\le\alpha$.
\item for every $\alpha,\beta\in E$ such that $\alpha<\beta$ and $\beta$ covers $\alpha$,
then $\epsilon(\alpha)\mathcal{R}\epsilon(\beta)$.
\end{enumerate}
We call the elements of $E$ {\it pieces}, and we say that the piece $\beta$ is above the piece 
$\alpha$ when $\alpha\le \beta$.
\end{defn}

In this paper, we consider only $P=\mathbb{N}_{\ge1}$ or the subset of $\mathbb{N}_{\ge1}^{k}$ 
for some positive integer $k\ge2$.

We introduce some terminologies for the later purpose.
A piece $\alpha$ is called a maximal piece if there is no piece $\beta$ such that 
$\alpha<\beta$.
A heap with a unique maximal piece is called a pyramid.
Let $m$ be a unique maximal piece of a pyramid.
The pyramid is said to be a semi-pyramid if the columns to the 
left of $m$ contain no piece. 

We introduce a class of heaps consisting of segments.
Let $P=\mathbb{N}_{\ge1}$ and a piece is a segment $[a,b]$ such that 
$a\le b$. If $a=b$ we call the segment a {\it monomer}.
If $b=a+1$, we call it a {\it dimer}.
The size of a segment $[a,b]$ is defined as $b-a$.
We drop segments from the top until it hits the horizontal line $y=0$ or 
it hits another segment. 
The two segments $[a_i,b_i]$, $1\le i\le 2$, have a relation 
if $[a_1,b_{1}]\cap[a_2,b_{2}]\neq\emptyset$. 
For example, Figure \ref{fig:HPstair} shows an example of 
a heap of segments of size $1$ and $2$.
When two segments have no relation (see $[1,2]$ and $[3,5]$ in Figure \ref{fig:HPstair}),
the order of dropping from top is irrelevant.
This means that these two segments are commutative.

Following \cite{Vie17}, we consider a decomposition of a heap via staircases.
Let $[a_i,b_i]$, $1\le i\le r$, be segments 
such that $b_{i}=a_{i+1}$ for all $1\le i\le r-1$ and 
$[a_{i},b_{i}]$ is just above $[a_{i+1},b_{i+1}]$, i.e., 
the segment $[a_{i},b_{i}]$ covers $[a_{i+1},b_{i+1}]$ in $E$.
We say that the set of these $r$ segments forms a {\it staircase}.
In the class of heaps consisting of segments, any heaps can 
be decomposed into staircases.
For this decomposition, we first take a largest staircase which contains the right-most 
maximal piece. 
Here, a maximal piece $[a_{\max},b_{\max}]$ means that there is no pieces $[c,d]$ such 
that $[c,d]$ covers $[a_{\max},b_{\max}]$ in $E$ and the segment $[c,d]$ is above $[a_{\max},b_{\max}]$.
Note that there may be several maximal pieces in a heap.  
Then, we remove the staircase which contains the right-most maximal piece 
from the heap, and continue to remove 
staircases one-by-one in the same way.
By construction, this decomposition is unique.
In Figure \ref{fig:HPstair}, we have two staircases:
\begin{align*}
S_{1}:=[1,2]\circ[2,4]\circ[4,5], \qquad S_{2}:=[3,5]\circ[5,6]\circ[6,7],
\end{align*}
where $[a,b]\circ[b,c]$ means the segment $[a,b]$ is above $[b,c]$ and 
these two segments form a staircase.
\begin{figure}[ht]
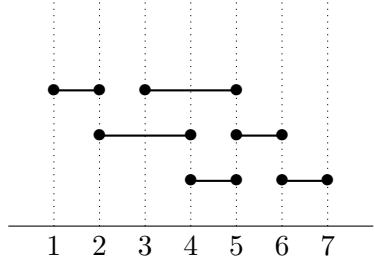

\tikzpic{-0.5}{[scale=0.6]
\draw(0,0)--(8,0);
\foreach \x in{1,2,...,7}{
\draw(\x,0)node[anchor=north]{$\x$};
\draw[dotted](\x,0)--(\x,5);
}
\draw[thick](4,1)node{$\bullet$}--(5,1)node{$\bullet$};
\draw[thick](2,2)node{$\bullet$}--(4,2)node{$\bullet$};
\draw[thick](1,3)node{$\bullet$}--(2,3)node{$\bullet$};
\draw[thick](6,1)node{$\bullet$}--(7,1)node{$\bullet$};
\draw[thick](5,2)node{$\bullet$}--(6,2)node{$\bullet$};
\draw[thick](3,3)node{$\bullet$}--(5,3)node{$\bullet$};
}
\caption{An example of a heap.}
\label{fig:HPstair}
\end{figure}
If the sizes of pieces depend only on the left abscissa of it,
we write a staircase as a segment $[a_1,b_{r}]$ by abuse of notation.
A staircase $[a_1,b_{r}]$ is uniquely written as a concatenation 
of segments:
\begin{align*}
[a_1,b_r]=[a_1,b_1]\circ[a_2,b_2]\circ\ldots\circ[a_r,b_r].
\end{align*}  
We write a heap in Figure \ref{fig:HPstair} as $S_{2}S_{1}$ which means that we 
first drop the staircase $S_1$ from the top and drop $S_{2}$ after $S_{1}$.
Note that this heap is different from the heap $S_{1}S_{2}$.
When two staircases have a concurrency relation as pieces, the order of 
the staircases plays an important role to express heaps in terms of 
staircases.

We introduce a valuation of a heap following \cite[Section 5]{Vie86}.
Let $\mathcal{H}(E)$ be the set of heaps of pieces from $E$.
Let $\mathbb{K}[[X]]$ be the algebra of formal power series with 
variables in a set $X$ and with coefficients in the 
commutative ring $\mathbb{K}$.
In this paper, we consider only $\mathbb{K}=\mathbb{C}$.
Given a piece in $E$, we define 
a {\it valuation} $w:E\rightarrow\mathbb{K}[[X]]$ which 
associates to the basic piece $p\in E$ a power series $v(p)$.
In this paper, we consider the case where $v(p)$ is a monomial in $X$.
The valuation $w(H)$ of a heap $H\in\mathcal{H}(E)$ is the product 
of the valuations of its pieces, i.e., $w(H)=\prod_{p\in H}w(p)$.

\begin{defn}
A trivial heap is a heap such that no pieces are above another.
We denote by $\mathcal{T}(E)$ the set of all trivial heaps with pieces from $E$.
\end{defn}

Let $\mathcal{M}$ be a subset of the pieces in $E$, and 
$\mathcal{H}(E;\mathcal{M})$
be the set of heaps with maximal pieces contained in $\mathcal{M}$.
The next proposition is an extension of inversion lemma (\cite[Proposition 5.1]{Vie86}).
\begin{prop}[{\cite[Corollary 4.5]{Kra06}}]
\label{prop:inversion}
We have 
\begin{align}
\label{eq:inversion}
\sum_{H\in\mathcal{H}(E;\mathcal{M})}w(H)
=\left(\sum_{T\in\mathcal{T}(E)}(-1)^{|T|}w(T)\right)^{-1}
\left(\sum_{T\in\mathcal{T}(E\setminus\mathcal{M})}(-1)^{|T|}w(T)\right),
\end{align}
where $\mathcal{T}(E)$ is the set of all trivial heaps.
\end{prop}

Note that the left hand side of (\ref{eq:inversion}) in Proposition \ref{prop:inversion} 
is the generating function for the heaps $H$.
We apply this proposition to the calculation of the several generating functions for heaps
of type $II$ and of symmetric Dyck paths in Sections \ref{sec:LPGF} and \ref{sec:symDP} respectively.

\section{Heaps of type \texorpdfstring{$I$}{I} for generalized Dyck paths}
\label{sec:HGDP}
\subsection{Lattice paths}
A {\it lattice path} is a path from $(0,0)$ to $(x,y)$ with $x,y\ge0$ 
which consists of up steps $(0,1)$ and right steps $(1,0)$. 
We denote an up (resp. right) step by $U$ (resp. $D$), 
and write a path $\mu$ as a word whose alphabets are in $\{U,D\}$.
For example, $\mu=UDDUUD$ corresponds to the following path:
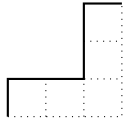
\begin{align*}
\tikzpic{-0.5}{[scale=0.5]
\draw[thick](0,0)--(0,1)--(2,1)--(2,3)--(3,3);
\draw[dotted](0,0)--(3,0)--(3,3)(1,0)--(1,1)(2,0)--(2,1)--(3,1)(2,2)--(3,2);
}
\end{align*}
There are ten lattice paths below $\mu=UDDUUD$.
By definition of a lattice path, these ten lattice paths are above 
the path $D^3U^3$.

A {\it peak} of a lattice path $\mu$ is a pattern $UD$ in $\mu$.
The path $UDDUUD$ has two peaks.
Similarly, a pattern $DU$ in $\mu$ is called a {\it valley}.
It is obvious that if we determine the position of peaks and valleys, 
a lattice path is uniquely determined.
Therefore, a lattice path is bijective to the sets of peaks and valleys.
Later in Sections \ref{sec:HGDP} and \ref{sec:HPLP}, 
we give a characterization of a lattice path by the positions 
of the peaks in the skew diagram which is surrounded by two lattice paths.
This implies that a lattice path is uniquely determined by 
the set of the positions of peaks. In this case, we do not need the 
positions of valleys since two peaks determine a valley uniquely.

A {\it rational Dyck path} (or simply {\it $(a,b)$-Dyck path}) of size $n$ is a lattice path from 
$(0,0)$ to $(bn,an)$ which is below 
the line $y=ax/b$ and above the line $y=0$.
For example, we have seven $(3,5)$-Dyck paths of size one. The highest path is $D^2UD^2UDU$ and 
the lowest path is $D^5U^3$.

We consider a special class of rational Dyck paths: $(1,b)$- and $(a,1)$-Dyck paths.
A $(1,b)$-Dyck path of size $n$ is a lattice path from $(0,0)$ to $(bn,n)$ below
$(UD^b)^n$ and above the line $y=0$.
Similarly, a $(a,1)$-Dyck path is a lattice path from $(0,0)$ to $(n,an)$ 
below $(U^aD)^{n}$ and above the line $y=0$. 
Let $D_1:=d_1\ldots d_{N}\in\{U,D\}^{\ast}$ be a $(a,b)$-Dyck path with $N=(a+b)n$ 
in terms of up and right steps.
We have a natural bijection between a $(a,b)$-Dyck path $D_1$ 
and $(b,a)$-Dyck path $D_{2}:=\bar{d_N}\bar{d_{N-1}}\ldots \bar{d_1}$
where $\bar{U}=D$ and $\bar{D}=U$. 
We say that the path $D_{2}$ is dual to the path $D_{1}$.

\subsection{Heaps for \texorpdfstring{$(1,b)$}{(1,b)}-Dyck paths}
\label{sec:1b}
We will construct heaps for $(1,b)$-Dyck paths of size $n$.
By definition, a $(1,b)$-Dyck path is a path below 
$(UD^{b})^{n}$ and above $D^{bn}U^{n}$.
When $n=0$, we have a unique path $\emptyset$.
We define a heap for $\emptyset$ as a trivial heap without pieces.

Suppose $n\ge1$.
We assign positive integers on the border of the path $D^{bn}U^{n}$
as follows.
We assign $bn-i$ to the point $(i,0)$, and $b(n+1-j)+1$ to the point $(bn,j)$
where $0\le i\le bn-1$ and $1\le j\le n$.

A peak in the lattice path $\mu$ is the pattern of consecutive steps $UD$.
Let $\mathtt{Peak}(\mu)$ be the set of peaks in $\mu$ and  $(i,j)\in\mathtt{Peak}(\mu)$ be the 
coordinate of the peak.
We read two labels $l_{i}$ and $r_{j}$ at the positions $(i,0)$ and $(bn,j)$.
By construction, we have $l_{i}<r_{j}$ and $r_{j}=1+b(n+1-j)$.
Define the set
\begin{align*}
\mathcal{S}(\mu):=\{[l_i,r_{j}] | (i,j)\in\mathtt{Peak}(\mu) \}.
\end{align*}

We construct a heap of pieces from $\mathcal{S}(\mu)$ as follow.
Formally, we take $P=\mathbb{N}_{\ge1}$, and $E$ be the set of segments 
of length $b'$ where $1\le b'\le b$.
We impose the following two conditions on a heap:
\begin{enumerate}[(B1)]
\item 
The length $l(p)$ of a piece $p$ at position $i$ satisfies
\begin{align*}
l(p)=b+1-i',
\end{align*}
where $i'\equiv i\pmod{b}$ and $i'\in[1,b]$.
\item There is no pair of two staircases $[l_i,r_j]$ and $[l_{i'},r_{j'}]$ 
such that 
\begin{align*}
l_{i}\le l_{i'}<r_{j'}\le r_{j}.
\end{align*}
\item There are at most $n$ staircases.
A staircase $[l,r]$ satisfies $1\le l\le r\le nb+1$.
\end{enumerate}

We introduce a linear order on the set $\mathcal{S}(\mu)$.
Given two staircases $[l_i,r_{j}]$ and $[l_{i'},r_{j'}]$, we write
\begin{align*}
[l_i,r_{j}]<[l_{i'},r_{j'}] \quad \text{ if } l_i<l_{i'}.
\end{align*}
Since a staircase $[l_i,r_j]$ corresponds to a peak in $\mu$, 
all $l_i$'s in $\mathcal{S}(\mu)$ are distinct.
This is compatible with the condition (B2).
We impose the condition (B3) since a peak corresponds to a staircase, and 
there are at most $n$ peaks in a $(1,b)$-Dyck path of size $n$.

Each segment $[l_i,r_j]$ corresponds to a staircase whose maximal 
piece is placed at the position $l_i$ and its total length is $r_{j}-l_{i}$.
Since the length of a piece depends only on its position, a staircase can be decomposed 
into a unique concatenation of pieces.
To have a heap, we pile a staircase according to the order of elements in $\mathcal{S}(\mu)$.
In this way, we obtain a heap of pieces corresponding to $\mathcal{S}(\mu)$.

\begin{defn}
We call a heap satisfying conditions (B1), (B2) and (B3) a heap of type $I$ for 
$(1,b)$-Dyck paths.
We denote by $\mathcal{H}_{(1,b)}(n)$ be the set of heaps of type $I$
for $(1,b)$-Dyck paths of size $n$.
\end{defn}

For example, we have twelve heaps for $b=2$ and $n=2$ as in Figure \ref{fig:HPk2}.
\begin{figure}[ht]
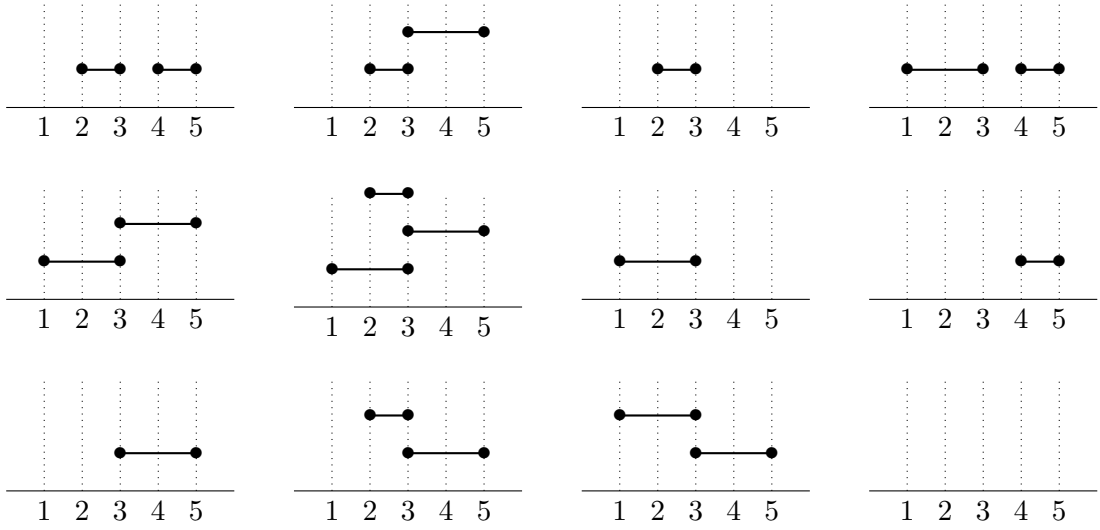

\tikzpic{-0.5}{[scale=0.5]
\draw(0,0)--(6,0);
\foreach \x in {1,2,3,4,5}{
\draw[dotted](\x,0)node[anchor=north]{$\x$}--(\x,3);	
}
\draw[thick](2,1)node{$\bullet$}--(3,1)node{$\bullet$}(4,1)node{$\bullet$}--(5,1)node{$\bullet$};
}\quad 
\tikzpic{-0.5}{[scale=0.5]
\draw(0,0)--(6,0);
\foreach \x in {1,2,3,4,5}{
\draw[dotted](\x,0)node[anchor=north]{$\x$}--(\x,3);	
}
\draw[thick](2,1)node{$\bullet$}--(3,1)node{$\bullet$}(3,2)node{$\bullet$}--(5,2)node{$\bullet$};
}\quad 
\tikzpic{-0.5}{[scale=0.5]
\draw(0,0)--(6,0);
\foreach \x in {1,2,3,4,5}{
\draw[dotted](\x,0)node[anchor=north]{$\x$}--(\x,3);	
}
\draw[thick](2,1)node{$\bullet$}--(3,1)node{$\bullet$};
}\quad 
\tikzpic{-0.5}{[scale=0.5]
\draw(0,0)--(6,0);
\foreach \x in {1,2,3,4,5}{
\draw[dotted](\x,0)node[anchor=north]{$\x$}--(\x,3);	
}
\draw[thick](1,1)node{$\bullet$}--(3,1)node{$\bullet$}(4,1)node{$\bullet$}--(5,1)node{$\bullet$};
}\\[11pt]
\tikzpic{-0.5}{[scale=0.5]
\draw(0,0)--(6,0);
\foreach \x in {1,2,3,4,5}{
\draw[dotted](\x,0)node[anchor=north]{$\x$}--(\x,3);	
}
\draw[thick](1,1)node{$\bullet$}--(3,1)node{$\bullet$}(3,2)node{$\bullet$}--(5,2)node{$\bullet$};
}\quad 
\tikzpic{-0.5}{[scale=0.5]
\draw(0,0)--(6,0);
\foreach \x in {1,2,3,4,5}{
\draw[dotted](\x,0)node[anchor=north]{$\x$}--(\x,3);	
}
\draw[thick](1,1)node{$\bullet$}--(3,1)node{$\bullet$}(2,3)node{$\bullet$}--(3,3)node{$\bullet$}
(3,2)node{$\bullet$}--(5,2)node{$\bullet$};
}\quad 
\tikzpic{-0.5}{[scale=0.5]
\draw(0,0)--(6,0);
\foreach \x in {1,2,3,4,5}{
\draw[dotted](\x,0)node[anchor=north]{$\x$}--(\x,3);	
}
\draw[thick](1,1)node{$\bullet$}--(3,1)node{$\bullet$};
}\quad 
\tikzpic{-0.5}{[scale=0.5]
\draw(0,0)--(6,0);
\foreach \x in {1,2,3,4,5}{
\draw[dotted](\x,0)node[anchor=north]{$\x$}--(\x,3);	
}
\draw[thick](4,1)node{$\bullet$}--(5,1)node{$\bullet$};
}\\[11pt]
\tikzpic{-0.5}{[scale=0.5]
\draw(0,0)--(6,0);
\foreach \x in {1,2,3,4,5}{
\draw[dotted](\x,0)node[anchor=north]{$\x$}--(\x,3);	
}
\draw[thick](3,1)node{$\bullet$}--(5,1)node{$\bullet$};
}\quad 
\tikzpic{-0.5}{[scale=0.5]
\draw(0,0)--(6,0);
\foreach \x in {1,2,3,4,5}{
\draw[dotted](\x,0)node[anchor=north]{$\x$}--(\x,3);	
}
\draw[thick](2,2)node{$\bullet$}--(3,2)node{$\bullet$}(3,1)node{$\bullet$}--(5,1)node{$\bullet$};
}\quad 
\tikzpic{-0.5}{[scale=0.5]
\draw(0,0)--(6,0);
\foreach \x in {1,2,3,4,5}{
\draw[dotted](\x,0)node[anchor=north]{$\x$}--(\x,3);	
}
\draw[thick](1,2)node{$\bullet$}--(3,2)node{$\bullet$}(3,1)node{$\bullet$}--(5,1)node{$\bullet$};
}\quad 
\tikzpic{-0.5}{[scale=0.5]
\draw(0,0)--(6,0);
\foreach \x in {1,2,3,4,5}{
\draw[dotted](\x,0)node[anchor=north]{$\x$}--(\x,3);	
}
} 
\caption{Twelve heaps of type $I$ for $(1,2)$-Dyck paths and $n=2$.}
\label{fig:HPk2}
\end{figure}
The heaps in the second row in Figure \ref{fig:HPk2} corresponds to 
the following $(1,2)$-Dyck paths:
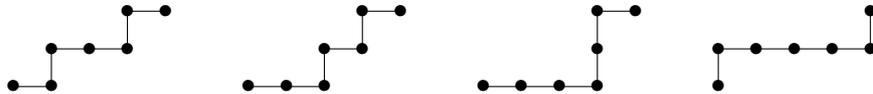
\begin{align*}
\tikzpic{-0.5}{[scale=0.5]
\draw(0,0)node{$\bullet$}--(1,0)node{$\bullet$}--(1,1)node{$\bullet$}
--(2,1)node{$\bullet$}--(3,1)node{$\bullet$}--(3,2)node{$\bullet$}--(4,2)node{$\bullet$};
}
\quad
\tikzpic{-0.5}{[scale=0.5]
\draw(0,0)node{$\bullet$}--(1,0)node{$\bullet$}--(2,0)node{$\bullet$}
--(2,1)node{$\bullet$}--(3,1)node{$\bullet$}--(3,2)node{$\bullet$}--(4,2)node{$\bullet$};
}
\quad
\tikzpic{-0.5}{[scale=0.5]
\draw(0,0)node{$\bullet$}--(1,0)node{$\bullet$}--(2,0)node{$\bullet$}
--(3,0)node{$\bullet$}--(3,1)node{$\bullet$}--(3,2)node{$\bullet$}--(4,2)node{$\bullet$};
}
\quad
\tikzpic{-0.5}{[scale=0.5]
\draw(0,0)node{$\bullet$}--(0,1)node{$\bullet$}--(1,1)node{$\bullet$}
--(2,1)node{$\bullet$}--(3,1)node{$\bullet$}--(4,1)node{$\bullet$}--(4,2)node{$\bullet$};
}
\end{align*}
The staircases in heaps correspond to peaks in the $(1,2)$-Dyck paths.
The second heap in the second row of Figure \ref{fig:HPk2} has two 
staircases: $[1,3]$ and $[2,5]$.
These two staircases correspond to the peaks at positions $(3,2)$ and $(2,1)$
in the second $(1,2)$-Dyck path.

To consider a generating function of heaps, we introduce a valuation 
on a heap.
\begin{defn}
\label{defn:valH}
Let $H$ be a heap.
We define a valuation $v(H)$ as 
\begin{align*}
v(H)&:=x^{(\text{number of staircases})}y^{\sum(\text{lengths of pieces of $H$})} \\
&\quad\times p^{\sum(\text{left abscissae of staircases of $H$})}
q^{\sum(\text{right abscissae of staircases of $H$})}.
\end{align*}
\end{defn}

With the valuation above, we introduce a generating function.
\begin{defn}
We define an ordinary generating function $F_{(1,b)}(x,y,p,q;n)$ as
\begin{align*}
F_{(1,b)}(x,y,p,q;n)
:=\sum_{H\in\mathcal{H}_{(1,b)}(n)}v(H).
\end{align*}
\end{defn}

For example, we have 
\begin{align*}
F_{(1,2)}(x,y,p,q;1)=1+xypq^3(p+y).
\end{align*}
since the three $(1,2)$-Dyck paths $DDU$, $DUD$ and $UDD$ give the valuations 
$1$, $xy^2pq^3$ and $xyp^2q^3$ respectively.

\subsection{Heaps for \texorpdfstring{$(a,1)$-}{(a,1)}-Dyck paths}
We construct heaps for $(a,1)$-Dyck paths of size $n$.
An $(a,1)$-Dyck path is a path below $(U^{a}D)^{n}$ 
and above $D^{n}U^{an}$.
As in the case of $(1,b)$-Dyck paths, we define a heap 
for the trivial Dyck path $\emptyset$ as the heap without pieces.

Suppose $n\ge1$. We assign positive integers on the border of the 
lowest path as follows.
We assign $l_{i}:=a(n-1-i)+1$ to the point $(i,0)$, and $r_{j}:=an+2-j$ to the 
point $(n,j)$  where $0\le i\le n-1$ and $1\le j\le an$.

Given an $(a,1)$-Dyck path $\nu$, we construct the set of segment $\mathcal{S}(\nu)$
as in the case of $(1,b)$-Dyck paths. 
Namely, we have 
\begin{align*}
\mathcal{S}(\nu):=\{[l_i,r_j] | (i,j)\in\mathtt{Peak}(\nu)\}.
\end{align*}

We take $P=\mathbb{N}_{\ge1}$ and $E$ be the set of dimers.
We impose three conditions on a heap:
\begin{enumerate}[(C1)]
\item The length of a piece is one.
\item The left abscissae of maximal pieces are $1$ modulo $a$.
\item A heap satisfies the condition (B2) and (B3).
\end{enumerate}

Given $\mathcal{S}(\nu)$, one can construct a heap as in the case of $(1,b)$-Dyck paths.
Since a piece is a dimer, any segment can be expressed as a staircase consisting of 
dimers.

\begin{defn}
A heap satisfying the conditions (C1), (C2) and (C3) is called 
a heap of type $I$ for $(a,1)$-Dyck paths.
We denote by $\mathcal{H}_{(a,1)}(n)$ the set of heaps of type $I$ 
for $(a,1)$-Dyck paths.
\end{defn}

For example, we have twelve heaps for $(2,1)$-Dyck paths of size two as in
Figure \ref{fig:HPk2d}.

\begin{figure}[ht]
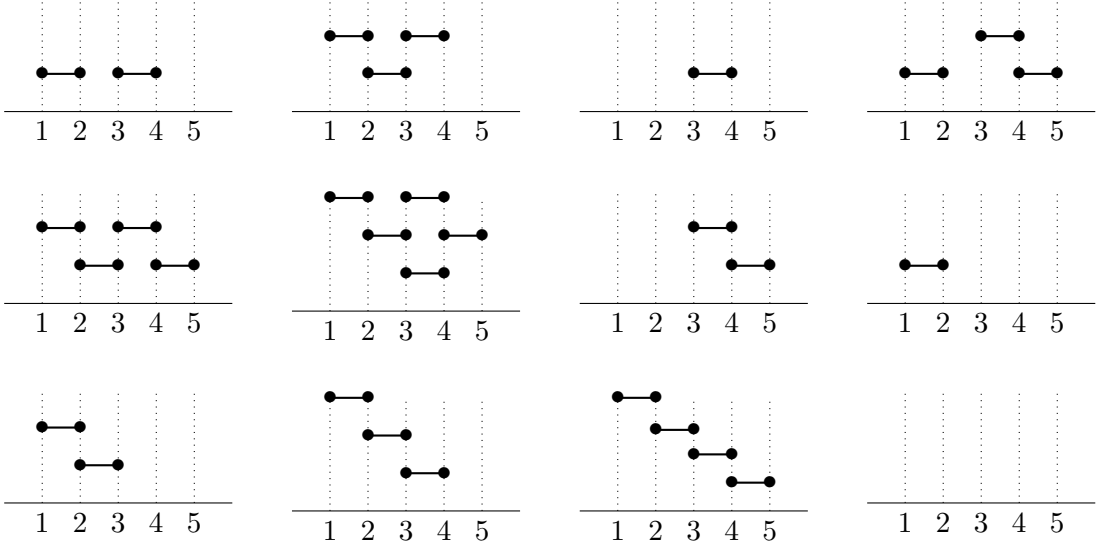

\tikzpic{-0.5}{[scale=0.5]
\draw(0,0)--(6,0);
\foreach \x in {1,2,3,4,5}{
\draw[dotted](\x,0)node[anchor=north]{$\x$}--(\x,3);	
}
\draw[thick](1,1)node{$\bullet$}--(2,1)node{$\bullet$}(3,1)node{$\bullet$}--(4,1)node{$\bullet$};
}\quad 
\tikzpic{-0.5}{[scale=0.5]
\draw(0,0)--(6,0);
\foreach \x in {1,2,3,4,5}{
\draw[dotted](\x,0)node[anchor=north]{$\x$}--(\x,3);	
}
\draw[thick](1,2)node{$\bullet$}--(2,2)node{$\bullet$}(2,1)node{$\bullet$}--(3,1)node{$\bullet$}
(3,2)node{$\bullet$}--(4,2)node{$\bullet$};
}\quad
\tikzpic{-0.5}{[scale=0.5]
\draw(0,0)--(6,0);
\foreach \x in {1,2,3,4,5}{
\draw[dotted](\x,0)node[anchor=north]{$\x$}--(\x,3);	
}
\draw[thick](3,1)node{$\bullet$}--(4,1)node{$\bullet$};
}\quad 
\tikzpic{-0.5}{[scale=0.5]
\draw(0,0)--(6,0);
\foreach \x in {1,2,3,4,5}{
\draw[dotted](\x,0)node[anchor=north]{$\x$}--(\x,3);	
}
\draw[thick](1,1)node{$\bullet$}--(2,1)node{$\bullet$}(3,2)node{$\bullet$}--(4,2)node{$\bullet$}
(4,1)node{$\bullet$}--(5,1)node{$\bullet$};
}\\[11pt]
\tikzpic{-0.5}{[scale=0.5]
\draw(0,0)--(6,0);
\foreach \x in {1,2,3,4,5}{
\draw[dotted](\x,0)node[anchor=north]{$\x$}--(\x,3);	
}
\draw[thick](1,2)node{$\bullet$}--(2,2)node{$\bullet$}(2,1)node{$\bullet$}--(3,1)node{$\bullet$}
(3,2)node{$\bullet$}--(4,2)node{$\bullet$}(4,1)node{$\bullet$}--(5,1)node{$\bullet$};
}\quad 
\tikzpic{-0.5}{[scale=0.5]
\draw(0,0)--(6,0);
\foreach \x in {1,2,3,4,5}{
\draw[dotted](\x,0)node[anchor=north]{$\x$}--(\x,3);	
}
\draw[thick](1,3)node{$\bullet$}--(2,3)node{$\bullet$}(2,2)node{$\bullet$}--(3,2)node{$\bullet$}
(3,1)node{$\bullet$}--(4,1)node{$\bullet$}(3,3)node{$\bullet$}--(4,3)node{$\bullet$}
(4,2)node{$\bullet$}--(5,2)node{$\bullet$};
}\quad 
\tikzpic{-0.5}{[scale=0.5]
\draw(0,0)--(6,0);
\foreach \x in {1,2,3,4,5}{
\draw[dotted](\x,0)node[anchor=north]{$\x$}--(\x,3);	
}
\draw[thick](3,2)node{$\bullet$}--(4,2)node{$\bullet$}(4,1)node{$\bullet$}--(5,1)node{$\bullet$};
}\quad 
\tikzpic{-0.5}{[scale=0.5]
\draw(0,0)--(6,0);
\foreach \x in {1,2,3,4,5}{
\draw[dotted](\x,0)node[anchor=north]{$\x$}--(\x,3);	
}
\draw[thick](1,1)node{$\bullet$}--(2,1)node{$\bullet$};
}\\[11pt]
\tikzpic{-0.5}{[scale=0.5]
\draw(0,0)--(6,0);
\foreach \x in {1,2,3,4,5}{
\draw[dotted](\x,0)node[anchor=north]{$\x$}--(\x,3);	
}
\draw[thick](1,2)node{$\bullet$}--(2,2)node{$\bullet$}(2,1)node{$\bullet$}--(3,1)node{$\bullet$};
}\quad 
\tikzpic{-0.5}{[scale=0.5]
\draw(0,0)--(6,0);
\foreach \x in {1,2,3,4,5}{
\draw[dotted](\x,0)node[anchor=north]{$\x$}--(\x,3);	
}
\draw[thick](1,3)node{$\bullet$}--(2,3)node{$\bullet$}(2,2)node{$\bullet$}--(3,2)node{$\bullet$}
(3,1)node{$\bullet$}--(4,1)node{$\bullet$};
}\quad 
\tikzpic{-0.5}{[scale=0.5]
\draw(0,0)--(6,0);
\foreach \x in {1,2,3,4,5}{
\draw[dotted](\x,0)node[anchor=north]{$\x$}--(\x,3);	
}
\draw[thick](1,3)node{$\bullet$}--(2,3)node{$\bullet$}(2,2.15)node{$\bullet$}--(3,2.15)node{$\bullet$}
(3,1.5)node{$\bullet$}--(4,1.5)node{$\bullet$}(4,0.75)node{$\bullet$}--(5,0.75)node{$\bullet$};
}\quad 
\tikzpic{-0.5}{[scale=0.5]
\draw(0,0)--(6,0);
\foreach \x in {1,2,3,4,5}{
\draw[dotted](\x,0)node[anchor=north]{$\x$}--(\x,3);	
}
}
\caption{Heaps of type $I$ for $(2,1)$-Dyck paths of size $2$.}
\label{fig:HPk2d}
\end{figure}
The $(2,1)$-Dyck paths for the heaps in Figure \ref{fig:HPk2d} 
are the dual of $(1,2)$-Dyck paths for the heaps in Figure \ref{fig:HPk2}.

As in the case of $(1,b)$-Dyck paths, we define a generating 
function for the heaps for $(a,1)$-Dyck paths.
\begin{defn}
We define an ordinary generating function $F_{(a,1)}(x,y,p,q;n)$
as 
\begin{align*}
F_{(a,1)}(x,y,p,q;n):=\sum_{H\in\mathcal{H}_{(a,1)}(n)}v(H),
\end{align*}
where the valuation $v(H)$ is defined in Definition \ref{defn:valH}.
\end{defn}

For example, we have
\begin{align*}
F_{(2,1)}(x,y,p,q;1)=1+xypq^2(1+yq),
\end{align*}
since three $(2,1)$-Dyck paths $DUU$, $UDU$ and $UUD$ give 
the valuations $1$, $xy^2pq^3$ and $xypq^2$ respectively.

\subsection{Duality between heaps of type \texorpdfstring{$I$}{I} for \texorpdfstring{$(1,a)$}{(1,a)}- and 
\texorpdfstring{$(a,1)$}{(a,1)}-Dyck paths}
\label{sec:dual1bb1}

Let $\mathcal{H}$ be a heap for a $(1,a)$-Dyck path of size $n$.
We decompose $\mathcal{H}$ into several staircases $\mathcal{H}_{i}$,
$1\le i\le r$, for some $r\in\mathbb{N}$.
Every staircase $\mathcal{H}_{i}$ can be expressed as the set $[l_{i},r_{i}]$
where $l_{i}$ (resp. $r_{i}$) are the left (resp. right) abscissae of the pieces 
in $\mathcal{H}_{i}$. 
Since the size of a $(1,a)$-Dyck path is $n$, the right-most abscissa of a piece is 
at most $an+1$.

We construct a map $\eta: \mathcal{H}\mapsto\mathcal{H}'$ where $\mathcal{H}'$ 
is a heap for a $(a,1)$-Dyck path. 

Take a right-most staircase $\mathcal{H}_{i}$ such that there are no pieces above $\mathcal{H}_{i}$.
We define a bar operation for a positive integer $p$ as
$\overline{p}:=2+an-p$. 
We place pieces corresponding to the staircase $[\overline{r_{i}},\overline{l}_{i}]$ in $\mathcal{H}'_{i}$.
Then, take a right-most staircase $\mathcal{H}_{j}\in\mathcal{H}\setminus\mathcal{H}_{i}$ such that 
there are no heaps above $\mathcal{H}_{j}$.
We place pieces corresponding to the staircase $[\overline{r_{j}},\overline{l_{j}}]$.
In this way, we remove staircases one-by-one from $\mathcal{H}$ and 
add staircases one-by-one to $\mathcal{H}'$.

To show that the map $\eta$ is well-defined, we have to show that the left abscissae of 
the staircases in $\mathcal{H}'$ take values in $\{1+aj: 0\le j\le n-1\}$.
Recall that the right abscissa $r_{i}$ of the staircase $[l_i,r_i]$ for a heap 
$\mathcal{H}$ takes a value in $\{1+a(n+1-j): 1\le j\le n\}$.
Since the left abscissa of the segment $[\bar{r_{i}},\bar{l_i}]$ in the heap $\mathcal{H}'$
is given by $\bar{r_{i}}$, it is easy to see that $\bar{r_i}$ takes a  value in 
$\{1+aj:0\le j\le n-1\}$. From these, $\eta$ is well-defined.

As a summary, we have the following proposition.
\begin{prop}
The map $\eta$ is a bijection between heaps $\mathcal{H}$ and $\mathcal{H}'$.
\end{prop}
\begin{proof}
It is obvious that the map $\eta$ is injective. 
Note that $\overline{\overline{p}}=p$.
We have a one-to-one correspondence between the two staircases $[l_i,r_{j}]$ 
and $[\overline{r_{i}},\overline{l_{i}}]$.
Suppose that $[\overline{r_{i}},\overline{l_{i}}]$ is a staircase in $\mathcal{H}'$.
One can construct a heap by the map $\eta':\mathcal{H}'\mapsto \mathcal{H}$.
The map $\eta'$ reverses the order of staircases and the positions of a staircase as in 
the case of $\eta$.
Then, $\eta'$ is injective. 
It is obvious that $\eta'\circ\eta:\mathcal{H}\mapsto \mathcal{H}$ and 
$\eta\circ\eta':\mathcal{H}'\mapsto \mathcal{H}'$ are identity.
From these, we have $\eta'=\eta^{-1}$, which implies that $\eta$ is a bijection. 
\end{proof}

\begin{example}
The heaps in Figure \ref{fig:HPk2} are dual to the heaps 
in Figure \ref{fig:HPk2d}.
\end{example}

\begin{example}
\label{ex:1331}
Consider the following $(1,3)$-Dyck path:
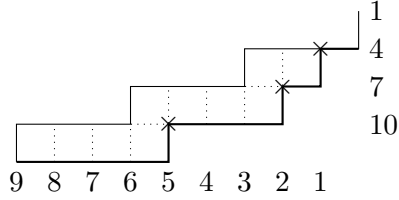
\begin{align*}
\tikzpic{-0.5}{[scale=0.5]
\draw(0,0)--(0,1)--(3,1)--(3,2)--(6,2)--(6,3)--(9,3)--(9,4);
\draw[thick](0,0)--(4,0)--(4,1)--(7,1)--(7,2)--(8,2)--(8,3)--(9,3);
\draw[dotted](1,0)--(1,1)(2,0)--(2,1)(3,0)--(3,1)--(4,1)--(4,2)
(5,1)--(5,2)(6,1)--(6,2)--(7,2)--(7,3);
\foreach \x in{1,2,3,4,5,6,7,8,9}{
\draw(9-\x,0)node[anchor=north]{$\x$};
}
\foreach \x/\y in {1/10,2/7,3/4,4/1}{
\draw(9,\x)node[anchor=west]{$\y$};
}
\draw(4,1)node{$\times$}(7,2)node{$\times$}(8,3)node{$\times$};
}
\end{align*}
The heap corresponding to this $(1,3)$-Dyck path has three staircases
$[1,4]$, $[2,7]$ and $[5,10]$.
The heaps are depicted as follows.
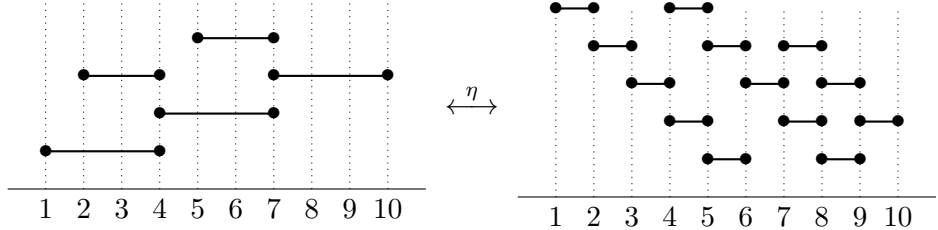
\begin{align*}
\tikzpic{-0.5}{[scale=0.5]
\draw(0,0)--(11,0);
\foreach \x in {1,2,...,10}{
\draw[dotted](\x,0)node[anchor=north]{$\x$}--(\x,5);
}
\draw[thick](1,1)node{$\bullet$}--(4,1)node{$\bullet$}(4,2)node{$\bullet$}--(7,2)node{$\bullet$}
(2,3)node{$\bullet$}--(4,3)node{$\bullet$}
(7,3)node{$\bullet$}--(10,3)node{$\bullet$}(5,4)node{$\bullet$}--(7,4)node{$\bullet$};
}
\overset{\eta}{\longleftrightarrow}
\tikzpic{-0.5}{[scale=0.5]
\draw(0,0)--(11,0);
\foreach \x in {1,2,...,10}{
\draw[dotted](\x,0)node[anchor=north]{$\x$}--(\x,5);
}
\draw[thick](1,5)node{$\bullet$}--(2,5)node{$\bullet$}(2,4)node{$\bullet$}--(3,4)node{$\bullet$}
(3,3)node{$\bullet$}--(4,3)node{$\bullet$}(4,2)node{$\bullet$}--(5,2)node{$\bullet$}
(5,1)node{$\bullet$}--(6,1)node{$\bullet$}(4,5)node{$\bullet$}--(5,5)node{$\bullet$}
(5,4)node{$\bullet$}--(6,4)node{$\bullet$}(6,3)node{$\bullet$}--(7,3)node{$\bullet$}
(7,2)node{$\bullet$}--(8,2)node{$\bullet$}(8,1)node{$\bullet$}--(9,1)node{$\bullet$}
(7,4)node{$\bullet$}--(8,4)node{$\bullet$}(8,3)node{$\bullet$}--(9,3)node{$\bullet$}
(9,2)node{$\bullet$}--(10,2)node{$\bullet$};
}
\end{align*}
The left heap corresponds to the $(1,3)$-Dyck path, and the right heap 
corresponds to the dual $(3,1)$-Dyck path.
Note that the size of a piece in the left heap may be one, two, or three. 
On the other hand, the size of a piece in the right heap is exactly one,
and the maximal pieces are at position $1,4$ and $7$.
The map $\eta$ sends the staircases $[1,4]$, $[2,7]$ and $[5,10]$ in the left heap
to the staircases $[7,10]$, $[4,9]$ and $[1,6]$ respectively. 
Note that the order of staircases is reversed. 
\end{example}

The duality of heaps can be translated in terms of generating functions 
as follows.

\begin{prop}
\label{prop:dual1}
Let $M=an+2$.
The generating functions satisfy 
\begin{align}
\label{eq:Fdual}
F_{(1,a)}(x,y,p,q)
=F_{(a,1)}(p^{M}q^{M}x,y,q^{-1},p^{-1}).
\end{align}

\end{prop}
\begin{proof}

Let $[i,j]$ be a segment in a heap of $\mathcal{H}_{(1,a)}$.
Then, this segment corresponds to a segment $[M-j,M-i]$ in the 
dual heap.
The valuation of $[i,j]$ is given by $xy^{j-i}p^{i}q^{j}$.
We consider the substitution: 
\begin{align*}
(x,y,p,q)\mapsto(p^{\alpha}q^{\alpha}x,y,q^{-1},p^{-1}).
\end{align*} 
Then, by solving the condition that the valuation $v([i,j])$ 
is equal to the valuation $v([M-j,M-i])$,
we obtain $\alpha=M$.
Note that the valuation of a heap is a product of the valuations
of segments.
From these observations, we have Eq. (\ref{eq:Fdual}).
\end{proof}

\begin{example}
Consider the same generalized Dyck path as in Example \ref{ex:1331}.
The weight given to heaps are 
$x^{3}y^{13}p^{8}q^{21}$ and $x^{3}y^{13}p^{12}q^{25}$ from left to right.
Then, we have 
\begin{align*}
(p^{11}q^{11}x)^{3}y^{13}q^{-12}p^{-25}=x^{3}y^{13}p^{8}q^{21}.
\end{align*}
\end{example}

\section{Heaps of type \texorpdfstring{$I$}{I} for lattice paths}
\label{sec:HPLP}
In Section \ref{sec:HGDP}, we introduce heaps for generalized 
Dyck paths.
In this section, we study heaps for lattice paths.
The lattice paths in this section may not be rational Dyck paths.
However, they possess similar properties to $(1,b)$- and $(a,1)$-Dyck 
paths at the same time.
We generalize the construction of heaps given in Section \ref{sec:HGDP}
by keeping its combinatorial structure.

Let $\mu:=\mu_{1}\mu_{2}\ldots\mu_{n}$ be a lattice path of length $n+1$ as a word of $\{U,D\}$ such 
that $\mu_{1}=U$ and $\mu_{n}=D$.
Let $N(U)$ and $N(D)$ be the number of up and right steps in $\mu$ respectively.
Then, the lowest path is given by $D^{N(D)}U^{N(U)}$.
We assign an integer label on the points $(i,0)$ with $0\le i\le N(D)-1$ and 
on the points $(N(D),j)$ with $1\le j\le N(U)$ as follows.
\begin{enumerate}[(D1)]
\item A label for the point $(i,0)$. 
Let $h(i)$ be the integer such that 
we have a right step connecting $(i,h(i))$ with $(i+1,h(i))$.
Then, we define an integer sequence  
$\mathfrak{s}':=(s'_0,s'_1,\ldots,s'_{N(D)-1})$ 
by $s'_{i}=N(U)+1-h(i)$.
By definition, we have $s'_{N(D)-1}=1$.
The sequence $\mathfrak{s}'$ is weakly decreasing from left to right.
We recursively construct a new integer sequence $\mathfrak{s}$ from $\mathfrak{s}'$.
\begin{enumerate}
\item Set $i=N(D)-2$.
\item If $s'_{i}\neq s'_{i+1}$, go to (c).
If $s'_{i}=s'_{i+1}$, then we increase all $s'_{j}$ with $j\le i$ by one, and go to (c)
\item If $i\neq 0$, decrease $i$ by one and go to (b). The algorithm stops when $i=0$.
\end{enumerate}
If we write the new integer sequence $\mathfrak{s}:=(s_0,s_1,\ldots, s_{N(D)-1})$, 
the label on the point $(i,0)$ is $s_{i}$.
\item A label for the point $(N(D),j)$. 
Suppose we have a peak in $\mu$ whose coordinate is $(i,j)$.
Let $J(\mu)$ be the set of $y$-coordinate of peaks.
Then, the label of the point $(N(D),j)$, $j\in J(\mu)$, is defined 
to be $s_{i}+1$.
For $j'\notin J(\mu)$, the label of the point $(N(D),j')$ is 
one plus the label of the point $(N(D),j'+1)$.
\end{enumerate}
We denote by $x(i)$ the label on $(i,0)$ and by $y(j)$ the label on $(N(D),j)$.

\begin{figure}[ht]
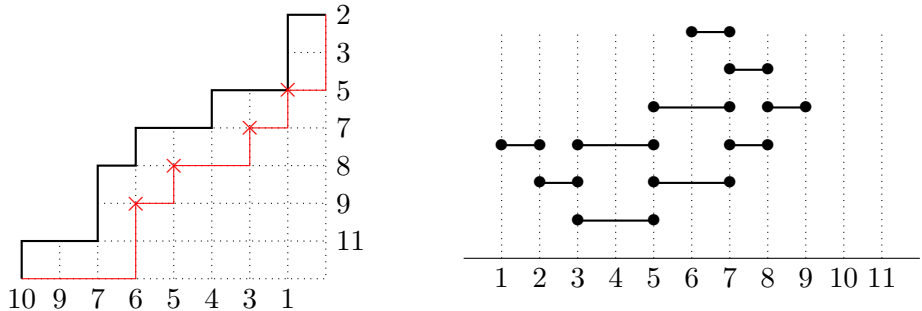

\tikzpic{-0.5}{[scale=0.5]
\draw[thick](0,0)--(0,1)--(2,1)--(2,3)--(3,3)--(3,4)--(5,4)--(5,5)--(7,5)--(7,7)--(8,7);
\draw[dotted](0,0)--(8,0)--(8,7)(1,0)--(1,1)(2,0)--(2,1)--(8,1)(2,2)--(8,2)
(3,0)--(3,3)--(8,3)(4,0)--(4,4)(5,0)--(5,4)--(8,4)(6,0)--(6,5)(7,0)--(7,5)--(8,5)(7,6)--(8,6);
\foreach \x/\y in {0/10,1/9,2/7,3/6,4/5,5/4,6/3,7/1}{
\draw(\x,0)node[anchor=north]{$\y$};
}
\foreach \x/\y in {1/11,2/9,3/8,4/7,5/5,6/3,7/2}{
\draw(8,\x)node[anchor=west]{$\y$};
}
\draw[red](0,0)--(3,0)--(3,2)node{$\times$}--(4,2)--(4,3)node{$\times$}--(6,3)--(6,4)node{$\times$}--(7,4)--
(7,5)node{$\times$}--(8,5)--(8,7);
}\qquad
\tikzpic{-0.5}{[scale=0.5]
\draw(0,0)--(12,0);
\foreach \x in {1,2,...,11}{
\draw(\x,0)node[anchor=north]{$\x$};
\draw[dotted](\x,0)--(\x,6);
}
\draw[thick](3,1)node{$\bullet$}--(5,1)node{$\bullet$}
(2,2)node{$\bullet$}--(3,2)node{$\bullet$}(1,3)node{$\bullet$}--(2,3)node{$\bullet$}
(5,2)node{$\bullet$}--(7,2)node{$\bullet$}(3,3)node{$\bullet$}--(5,3)node{$\bullet$}
(7,3)node{$\bullet$}--(8,3)node{$\bullet$}(5,4)node{$\bullet$}--(7,4)node{$\bullet$}
(8,4)node{$\bullet$}--(9,4)node{$\bullet$}(7,5)node{$\bullet$}--(8,5)node{$\bullet$}
(6,6)node{$\bullet$}--(7,6)node{$\bullet$};
}
\caption{A heap for the lattice path $D^3U^2DUD^2UDUDU^2$ below the path $UD^2U^2DUD^2UD^2U^2D$}
\label{fig:exlabel}
\end{figure}

We impose the following conditions on pieces:
\begin{enumerate}[(E1)]
\item A staircase $[l_i,r_i]$ satisfies $l_i=x(j)$ and $r_i=y(k)$ for some $j$ and $k$.
\item Let $d(j):=\#\{i\le j | s'_{i}=s'_{j} \}$ where $s'_{i}$ is defined in (D1). 
Then, the size of a piece starting from $x(i)$ is $d(j)$.
\item A heap satisfies the condition (B2).
\end{enumerate}

\begin{defn}
We call a heap satisfying (E1), (E2) and (E3) a heap of type $I$ for a lattice path $\mu$.
We denote by $\mathcal{H}^{I}(\mu)$ the set of heaps of type $I$ for $\mu$.
\end{defn}

When $\mu$ is a generalized Dyck path, a $(1,b)$-Dyck path or an $(a,1)$-Dyck path, 
a heap of type $I$ for $\mu$ is nothing but a heap of type $I$ for the generalized Dyck path 
considered in Section \ref{sec:HGDP}.

\begin{example}
Figure \ref{fig:exlabel} shows an example of a lattice path and its heap.
We have four staircases $[1,5]=[1,2]\circ[2,3]\circ[3,5]$, $[3,7]=[3,5]\circ[5,7]$, 
$[5,8]=[5,7]\circ[7,8]$ and $[6,9]=[6,7]\circ[7,8]\circ[8,9]$ which correspond 
to peaks.
Note that pieces starting from the third and fifth positions have size two.
There are no staircases starting from the second and eighth positions.
This is because we have no labels with two and eight in $\{y(j)\}$.
\end{example}

Let $\mu:=\mu_1\ldots\mu_{n}$ be a path expressed as a word of $\{U,D\}$.
We define $\overline{U}:=D$ and $\overline{D}:=U$.  
We define a dual path $\overline{\mu}:=\overline{\mu_{n}}\overline{\mu_{n-1}}\ldots\overline{\mu_1}$.
Let $\mathcal{P}(\mu)$ be the set of lattice paths below $\mu$.

Let $M(\mu)$ be the maximum label $\max\{y(j)\}$ for the lattice path $\mu$.

\begin{lemma}
\label{lemma:M}
We have $M(\mu)=M(\overline{\mu})$.
\end{lemma}
\begin{proof}
Let $\mu$ be a lattice path of the form $U^{b_{n}}D^{a_{n}}U^{b_{n-1}}D^{a_{n-1}}\ldots U^{b_{1}}D^{a_{1}}$.
We compute $M(\mu)$ by following its definition.
The number of valleys in $\mu$, i.e., the number of pattern $DU$, is $n-1$.
The maximum label $\max\{x(i)\}$ is given by 
$A=\sum_{i=1}^{n}a_{i}+\sum_{j=1}^{n-1}b_{n}-(n-2)$.
The label $B$ at the point $(\sum_{i=1}^{n}a_{n},b_{n})$ is given by $B=A+1$ by definition.
Then, the maximum label $M(\mu)$ is given by
$B+b_{n}-1$, which implies that $M(\mu)=\sum_{i=1}^{n}(a_{i}+b_{i})-(n-1)$.

If we take a dual of the path $\mu$, the roles of $a_{i}$ and $b_{i}$ are exchanged and 
the number of valleys is the same as before.
Therefore, we have $M(\mu)=M(\overline{\mu})$.
\end{proof}

\begin{remark}
In Section \ref{sec:HGDP}, we introduce a map $\eta$ to describe the duality between 
$(1,b)$-Dyck paths and $(b,1)$-Dyck paths.
This $\eta$ can be easily generalized to describe the duality between the paths in 
$\mathcal{P}(\mu)$ and $\mathcal{P}(\overline{\mu})$.
Namely, we decompose a heap into staircases, then construct a heap by reversing 
the order and the coordinates of the staircases.
\end{remark}

Recall that $\mathcal{H}^{I}(\mu)$ is the set of heaps for the lattice path $\mu$.
Then, we define the ordinary generating function by 
\begin{align*}
F_{\mu}(x,y,p,q):=\sum_{H\in\mathcal{H}^{I}(\mu)}v(H),
\end{align*}
where the valuation $v(H)$ is defined in Definition \ref{defn:valH}.

By the same proof as in Proposition \ref{prop:dual1}, we obtain 
the duality in terms of generating functions as follows.
\begin{prop}
Let $M=1+\max(y(j))$. Then, we have
\begin{align*}
F_{\mu}(x,y,p,q)=F_{\overline{\mu}}(p^{M}q^{M}x,y,q^{-1},p^{-1}).
\end{align*}
\end{prop}

Note that the duality is compatible with Lemma \ref{lemma:M}

\section{Heaps of type \texorpdfstring{$II$}{II} for lattice paths and generating functions}
\label{sec:LPGF}
\subsection{Heaps and generating functions}
\label{subsec:HGFII}
Let $\mu$ be a lattice path. 
Let $a_{i}$, $1\le i\le N(U)$, be the number of 
$D$ steps at height $N(U)+1-i$.
We define $a_{N(U)+1}:=1$.
An integer sequence $\mathfrak{a}(\mu):=(a_1,a_2\ldots,a_r)$
is defined from the numbers $a_{i}$ where $r=N(U)+1$.
We define the length $l(\mathfrak{a})$ of $\mathfrak{a}$ as $r$.
For example, the path $UD^2U^2DUD^2UD^2U^2D$ in Figure \ref{fig:exlabel} 
gives the sequence $\mathfrak{a}=(1,0,2,2,1,0,2,1)$.

We consider heaps $H$ satisfying the following conditions:
\begin{enumerate}[(F1)]
\item $H$ has $l(\mathfrak{a})=r$ pieces from top to bottom, 
and the length of $r$-th piece from top is $a_{r}$.
\item $H$ has a unique maximal piece whose left abscissa is one.
\end{enumerate}
Note that a piece of length zero corresponds to a monomer.
Formally, we consider heaps on the set $P=\mathbb{N}_{\ge1}$.
The concurrency relation $\mathcal{R}$ is defined by  
$a\mathcal{R}b$ if and only if $a\cap b\neq\emptyset$.
Let $E$ be the set of segments of the form $\alpha=[a,b]$ with 
$1\le a\le b$.  The map $\epsilon:E\rightarrow P$ is given by
$\epsilon(\alpha)=[a,b]$.

\begin{defn}
We call a heap satisfying the conditions (F1) and (F2) a heap of type $II$.
We define $\mathcal{H}^{II}(\mathfrak{a})$ as the set of heaps of type $II$.
\end{defn}

Let $\mu_{0}$ be a lattice path and $\mu$ be a lattice path below $\mu_{0}$.
We construct a heap of type $II$ from the pair $(\mu_{0},\mu)$ as follows.
Recall that the length $r$ of $\mathfrak{a}(\mu_{0})$ is one plus the number of 
up steps $U$ in $\mu_{0}$. From the condition (F1), we have $r$ pieces from top
to bottom.
By the condition (F2), the position of the unique 
maximal piece is uniquely fixed. 
From these, it is enough to fix the positions of the $r-1$ pieces in a heap.

Let $\mathbf{b}(\mu):=(b_1(\mu),\ldots,b_{r-1}(\mu))$ be a sequence of integers 
such that $b_{i}$ is the number of right steps left to the $i$-th up step
in $\mu$.
For example, we have $\mathbf{b}=(0,2,2)$ for the path $UD^2U^2D$.
Let $\mathbf{c}:=(c_1,\ldots,c_{r-1})$ be a sequence of integers 
given by 
$c_{i}=b_{r-i}(\mu)-b_{r-i}(\mu_{0})+1$ for $1\le i\le r-1$.
Then, we give the left abscissa of the $i+1$-th piece from the top by 
$c_{i}$ for $1\le i\le r-1$.

\begin{example}
Let $\mu_0=UD^2U^2D$ and $\mu=DUDUDU$.
By definitions, we have $\mathfrak{a}(\mu_{0})=(1,0,2,1)$, 
$\mathbf{b}(\mu_{0})=(0,2,2)$ and $\mathbf{b}(\mu)=(1,2,3)$.
From these we have $\mathbf{c}=(2,1,2)$.
From $\mathfrak{a}(\mu_{0})$ and $\mathbf{c}$, we have 
the following heap:
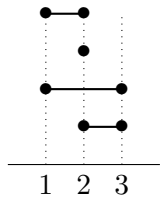
\begin{align*}
\tikzpic{-0.5}{[scale=0.5]
\draw(0,0)--(4,0);
\foreach \x in {1,2,3}{
\draw(\x,0)node[anchor=north]{$\x$};
}
\draw[thick](2,1)node{$\bullet$}--(3,1)node{$\bullet$}
(1,2)node{$\bullet$}--(3,2)node{$\bullet$}(2,3)node{$\bullet$}
(1,4)node{$\bullet$}--(2,4)node{$\bullet$};
\draw[dotted](1,0)--(1,4)(2,0)--(2,4)(3,0)--(3,4);
}
\end{align*} 
The lengths of the pieces from top to bottom are given by $\mathfrak{a}(\mu_{0})$,
and the left abscissae of the pieces are given by $\mathbf{c}$.
\end{example}

We have a heap constructed from the pair of lattice paths $(\mu_0,\mu)$ via
the sequences of integers $\mathfrak{a}(\mu_0)$ and $\mathbf{c}$.
We show that the heap is actually of type $II$.
First, by construction, the condition (F1) is satisfied.
Secondly, to show that this construction is admissible, we show that a heap constructed 
from $\mathfrak{a}(\mu_0)$ and $\mathbf{c}$ satisfies the condition (F2).
The condition (F2) is equivalent to the following condition:
\begin{align}
\label{eq:condac}
a_{i}+c_{i-1}\ge c_{i},\quad 1\le i\le r-1,
\end{align}
where $c_{0}=1$.
This condition can be obtained from the condition that the $i$-th piece is not 
a maximal piece in a heap of type $II$.
This condition is equivalent to the condition that the left abscissa of the $i$-th 
piece is weakly left to the right abscissa of the $i-1$-th piece.
However, the condition (\ref{eq:condac}) is satisfied by $\mathfrak{a}(\mu_{0})$ 
and $\mathbf{c}$ since the path $\mu$ is below the path $\mu_{0}$.
From these, the construction of heaps of type $II$ is well-defined.

Given a heap $H\in\mathcal{H}^{II}(\mathfrak{a})$, we define the valuation $w(H)$
by 
\begin{align*}
w(H):=p^{\sum(\text{left abscissae of pieces of $H$})}.
\end{align*}
Then, we define an ordinary generating function $G(\mathfrak{a})$ by 
\begin{align*}
G(\mathfrak{a}):=\sum_{H\in\mathcal{H}^{II}(\mathfrak{a})}w(H).
\end{align*}
Note that $\mathcal{H}^{II}(\mathfrak{a})$ is the set of heaps
characterized by the sequence $\mathfrak{a}$.
The number of pieces is $r$, and the length of a piece is 
given by $a_{i}$, $1\le i\le r$.
Since the number of pieces and the length of a piece are fixed by $\mathfrak{a}$,
the valuation $w(H)$ looks simple compared to the valuations for heaps of 
type $I$.

Below, we introduce the Lindstr\"om--Gessel--Viennot lemma \cite{GesVie85,Lin73} and 
give a determinant expression of $G(\mathfrak{a})$ as an application.

Let $X=\{x_1,\ldots,x_{n}\}$ and $Y=\{y_1,\ldots,y_{n}\}$ be the sets of 
vertices. Given two vertices $x$ and $y$, we denote 
by $w(x,y)$ the sum of the weights given to paths from $x$ to $y$.
If we assign the weight one to each path, the sum $w(x,y)$ counts the number of 
paths from $x$ to $y$.

A $n$-tuple of non-intersecting paths from $X$ to $Y$ is an $n$-tuple $(P_1,\ldots,P_n)$ 
such that 
\begin{enumerate}[(G1)]
\item There exists a permutation $\rho$ of $[n]$ such that $P_{i}$ is a path from $x_{i}$ to $y_{\rho(i)}$.
\item If $i\neq j$, the paths $P_i$ and $P_{j}$ do not occupy the same vertex.
\end{enumerate}
Let $\mathcal{P}:=(P_1,\ldots, P_n)$ and $\rho(P)$ be the permutation satisfying the condition (G1).
We denote by $w(P_{i}):=w(x_{i},y_{\rho(i)})$ the weight given to the path $P_{i}$ from $x_i$ to $y_{\rho(i)}$.

\begin{lemma}[\cite{GesVie85,Lin73}]
\label{lemma:LGV}
The weighted sum of non-intersecting paths from $X$ to $Y$ is given by 
the following determinant: 
\begin{align*}
\det\left[w(x_i,y_j)\right]_{1\le i,j\le n}=\sum_{\mathcal{P}:X\rightarrow Y}
\mathrm{sign}(\rho(P))\prod_{i=1}^{n}w(P_{i}).
\end{align*}
\end{lemma}

Given a sequence $\mathfrak{a}$, we define lattice points $s_{i}$ and $e_{i}$ for 
$1\le i\le r$ by 
\begin{align}
\label{eq:coordse}
s_{i}:=(-(i-1)-\sum_{p\le i}a_{p},i-1), \quad e_{i}:=(-(i-1),i).
\end{align}

Suppose that we have a sequence of up and right steps $(t_1,\ldots,t_{l})$ from 
$s_{i}$ to $e_{j}$.
Here, $l$ is the number of steps from $s_i$ to $e_j$.
We assign a weight for a each step $p^{-(k-1)}$ if $t_{k}$ is an up step and $1$ 
if $t_{k}$ is a right step. The weight of a path from $s_i$ to $e_{j}$ is 
the product of weights given to each step.
More precisely, we define the weight to a path from $s_{i}$ to $e_{j}$ by 
\begin{align*}
w(s_{i}\rightarrow e_{j})
:=\begin{cases}
\displaystyle p^{-((j-i)(j-i+1)/2)\delta(j>i)}\genfrac{[}{]}{0pt}{}{1+\sum_{p\le i}a_{p}}{j-i+1}_{p^{-1}}, 
& \text{ if  } i-1\le j\le i+\sum_{p<i}a_{p}, \\
0, & \text{otherwise},
\end{cases}
\end{align*}
where $\delta$ is the Kronecker delta function, and 
$\displaystyle\genfrac{[}{]}{0pt}{}{n}{m}_{p}:=\genfrac{}{}{}{}{[n]_{p}!}{[m]_{p}![n-m]_{p}!}$ with 
$[n]_{p}:=\sum_{i=0}^{n}p^{i}$.
Note that, in each path, we allow only up and right steps from $s_i$ to $e_{j}$.

By applying the Lindstr\"om--Gessel--Viennot lemma \ref{lemma:LGV}, we obtain
\begin{prop}
\label{prop:LGV}
We have 
\begin{align}
\label{eq:GdetLGV}
G(\mathfrak{a})=p^{D(\mathfrak{a})}\det\left[w(s_{i}\rightarrow e_{j})\right]_{1\le i,j\le r-1},
\end{align}
where $D(\mathfrak{a}):=\sum_{j=1}^{r-1}(r-j)a_{j}$.
\end{prop}
\begin{proof}
Given two sets of points $\{s_i\}$ and $\{e_{i}\}$, we denote an $n$-tuple of 
paths by $(P_1,\ldots,P_n)$.
Let $j$ be the minimum integer such that $P_{j}$ is not a path from $s_j$ to 
$e_{j}$.
Since each $s_{i}$ is connected to $e_{\rho(i)}$, there exists $i'$ such that 
$P_{i'}$ is a path from $s_{i'}$ to $e_{j}$.
By the minimal property of $j$, we have $i'>j$.
By the $y$-coordinate of $s_{i}$ is strictly larger than that of $e_{i}$, 
the paths $P_i$ and $P_{i'}$ intersect.
This implies that the right hand sight of Lemma \ref{lemma:LGV} can be 
reduced to the sum of $n$-tuple paths whose permutation $\rho(P)$ is 
the identity on $[n]$.

Since $\rho(P)$ is the identity, we consider only paths $P_i$ 
from $s_{i}$ to $e_{i}$.
The weight $w(s_{i}\rightarrow e_{i})$ is given by $[n']$ with some 
non-negative integer $n'$.
Given an integer $n''\in\{0,1,\ldots,n'-1\}$, we have a weight $p^{-n''}$ 
for a path from $s_{i}$ to $e_{i}$.

Recall that a lattice path $\mu$ is characterized by $\mathfrak{a}$.
Let $\mu'$ be a lattice path below $\mu$.
Then, the number of unit boxes above $\mu'$ and below $\mu$ 
is counted by the exponent of $p$. 

Combining these two observations, 
the integer $n''$ corresponds to the number of unit boxes in the 
$i$-th row below $\mu$ and above $\mu'$.
The expression (\ref{eq:GdetLGV}) follows.
\end{proof}

\begin{example}
We consider the path $\mu=UDU^2D^2$ and $\mu'=DUDUDU$.
\begin{figure}[ht]
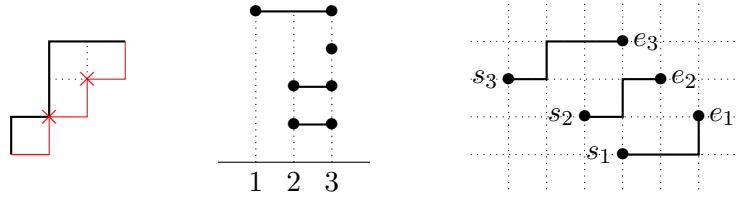

\tikzpic{-0.5}{[scale=0.5]
\draw[thick](0,0)--(0,1)--(1,1)--(1,3)--(3,3);
\draw[red](0,0)--(1,0)--(1,1)node{$\times$}--(2,1)--(2,2)node{$\times$}--(3,2)--(3,3);
\draw[dotted](1,2)--(2,2)--(2,3);
}\qquad
\tikzpic{-0.5}{[scale=0.5]
\draw(0,0)--(4,0);
\foreach \x in {1,2,3}{
\draw(\x,0)node[anchor=north]{$\x$};
}
\draw[thick](2,1)node{$\bullet$}--(3,1)node{$\bullet$}
(2,2)node{$\bullet$}--(3,2)node{$\bullet$}(3,3)node{$\bullet$}
(1,4)node{$\bullet$}--(3,4)node{$\bullet$};
\draw[dotted](1,0)--(1,4)(2,0)--(2,4)(3,0)--(3,4);
}\qquad
\tikzpic{-0.5}{[scale=0.5]
\draw(0,1)node{$\bullet$}node[anchor=west]{$e_{1}$};
\draw(-1,2)node{$\bullet$}node[anchor=west]{$e_{2}$};
\draw(-2,3)node{$\bullet$}node[anchor=west]{$e_{3}$};
\draw(-2,0)node{$\bullet$}node[anchor=east]{$s_{1}$};
\draw(-3,1)node{$\bullet$}node[anchor=east]{$s_{2}$};
\draw(-5,2)node{$\bullet$}node[anchor=east]{$s_{3}$};
\draw[thick](-2,0)--(0,0)--(0,1)(-3,1)--(-2,1)--(-2,2)--(-1,2)(-5,2)--(-4,2)--(-4,3)--(-2,3);
\foreach \x in{-5,...,0}{
\draw[dotted](\x,4)--(\x,-1);
}
\foreach \x in{0,...,3}{
\draw[dotted](-6,\x)--(1,\x);
}
}
\caption{The correspondence among the lattice path, the heap, and the non-intersecting paths}
\label{fig:corLGV}
\end{figure}
In Figure \ref{fig:corLGV}, we summarize the correspondence among the lattice path $\mu'$, 
the heap for $\mu'$, and non-intersecting lattice paths.
In the left picture, the numbers of unit boxes below $\mu$ and above $\mu'$ is 
$2, 1$, and $1$ from top to bottom.
In the middle picture, the numbers of empty sites left to pieces are $2$, $1$, $1$ 
from top to bottom.
In the right picture, the numbers of right steps in $s_{i}$ to $e_{i}$ which is 
left to the unique up step are $1, 1$ and $2$ from top to bottom.
\end{example}

Below, we give two applications of the lattice path method considered above to
the generating functions of subheaps. 

Let $\widetilde{G_{k}}(\mathfrak{a})$ be the generating function of heaps such that 
the right abscissa of the bottom piece is $k$ where $k\ge a_{r}$.
By definition, we have $G(\mathfrak{a})=\sum_{k}\widetilde{G_{k}}(\mathfrak{a})$.
The next proposition gives a recursive formula for $\widetilde{G_k}(\mathfrak{a})$.
\begin{prop}
The generating function $\widetilde{G_{k}}(\mathfrak{a})$ satisfies 
\begin{align}
\label{eq:Gtilrec}
\widetilde{G_{k}}(a_1,\ldots,a_{r})
=p\cdot\widetilde{G_{k-1}}(a_1,\ldots, a_r)-p^{k-a_{r}}\widetilde{G_{k-1-a_r}}(a_1,\ldots,a_{r-1}),
\end{align}
with initial conditions $\widetilde{G_{n}}(\mathfrak{a})=0$ for $n\le 0$, 
and $\widetilde{G_{m}}(\mathfrak{a})=p^{r+a_1+a_2+\ldots+a_{r-1}}$ for $m=1+\sum_{1\le i\le r} a_{i}$.
\end{prop}
\begin{proof}
By definition of a heap, the left abscissa of the $r$-th piece from top 
is weakly left to the right abscissa of the $r-1$-th piece.
This implies 
\begin{align}
\label{eq:Gtilde}
\widetilde{G_{k}}(\mathfrak{a})=p^{k-a_{r}}\sum_{j\ge k-a_r}\widetilde{G_{j}}(a_{1},\ldots,a_{r-1}).
\end{align}
By subtracting $p\cdot\widetilde{G_{k-1}}(\mathfrak{a})$ from Eq. (\ref{eq:Gtilde}), 
we obtain Eq. (\ref{eq:Gtilrec}).
The initial conditions are obvious.
\end{proof}

We give an expression of $\widetilde{G_{k}}$ in terms of weighted sum of 
$r-1$-tuple non-intersecting lattice paths.
\begin{defn}
We define $\widetilde{\mathcal{P}}(\{s_{i}\},\{e_{i}\}; M \})$ as the weighted sum of 
non-intersecting lattice paths from $\{s_{i}\}$ to $\{e_{i}\}$ for $1\le i\le r-1$ 
such that the points in $M$ are passed through by a $r-1$-th path. 
\end{defn}

\begin{prop}
Let $\{s_i\}$ and $\{e_{i}\}$ be the coordinates as in Eq. (\ref{eq:coordse}), 
and $M$ be the set of two points:
\begin{align*}
M=\{s_{r-1}+(k-1,0),s_{r-1}+(k-1,1)\}.
\end{align*}
Then, the generating function $\widetilde{G_{k}}(\mathfrak{a})$ is given by
\begin{align*}
\widetilde{G_{k}}(\mathfrak{a})=p^{D(\mathfrak{a})}\widetilde{\mathcal{P}}(\{s_{i}\},\{e_{i}\}; M \}).
\end{align*}
\end{prop}
\begin{proof}
As in the proof of Proposition \ref{prop:LGV}, we have a correspondence 
between a lattice path and $r$-tuple of non-intersecting paths.
Suppose that $r-1$-th non-intersecting path passes through the edge consisting of two points in $M$.
Then, this configuration of the path implies that the right abscissa of the $r$-th piece in a heap
is $k$.
This means that the number $\widetilde{\mathcal{P}}(\{s_{i}\},\{e_{i}\}; M \})$ is nothing but 
the number $\widetilde{G_{k}}$ of heaps.
This completes the proof. 
\end{proof}

The previous proposition can be easily generalized as follows.
\begin{prop} 
Let $I$ be the subset of $[1,r-1]$, and 
suppose that  the right abscissa of $i$-th piece of a heap $H$ is $k_{i}$
for $i\in I$.
The generating function of heaps $H$ satisfying the above conditions 
is given by $p^{D(\mathfrak{a})}\widetilde{\mathcal{P}}(\{s_{i}\},\{e_{i}\}; M' \})$
where 
\begin{align*}
M':=\{s_{i}+(k_{i}-a_{i+1},0),s_{i}+(k_{i}-a_{i+1},1) | i\in I \}.
\end{align*}
\end{prop}

We consider another generating functions of heaps.
Let $G_{k}(\mathfrak{a})$ be the generating function of heaps 
such that the maximal right abscissa of pieces is $k$.
By definition, we have $G(\mathfrak{a})=\sum_{k}G_{k}(\mathfrak{a})$.

The generating function $G_{k}(\mathfrak{a})$ can be expressed 
as an enumeration of lattice paths.
\begin{defn}
We define $\mathcal{P}(\{s_{i}\},\{e_{i}\};\{m_{i}\})$ as the weighted sum of 
non-intersecting lattice paths from $\{s_{i}\}$ to $\{e_{i}\}$ such that 
the $i$-th path from $s_i$ does not pass through the point $m_{i}$ for 
all $1\le i\le r-1$. 
\end{defn}

We define the points $\{m_{i}(k)\}$ for a given $k\ge\max\{a_i:1\le i\le r-1\}$ by 
\begin{align}
\label{eq:defm}
m_{i}(k):=s_i+(k-a_{i+1},0),
\end{align}
for $1\le i\le r-1$.

\begin{prop}
Let $\{s_i\}$ and $\{e_{i}\}$ be the coordinates as in Eq. (\ref{eq:coordse}), 
and $\{m_{i}(k)\}$ be defined as in Eq. (\ref{eq:defm}).
Then,
\begin{align}
\label{eq:GkPm}
G_{k}(\mathfrak{a})=p^{D(\mathfrak{a})}\bigg(
\mathcal{P}(\{s_{i}\},\{e_{i}\};\{m_{i}(k)\})-\mathcal{P}(\{s_{i}\},\{e_{i}\};\{m_{i}(k-1)\})
\bigg),
\end{align}
where we define $\mathcal{P}(\{s_{i}\},\{e_{i}\};\{m_{i}(k)\}):=0$ if $k<\max\{a_{i}:1\le i\le r-1\}$.
\end{prop}
\begin{proof}
Suppose that the $i$-th path of $r$-tuple non-intersecting paths does not
pass through the point $m_{i}(k)$.
This means that the right abscissa of the $i$-th piece in a heap 
is less than $k$.
The generating function $G_{k}(\mathfrak{a})$ is the weighted sum over all 
heaps whose maximal right abscissa of pieces is $k$.
The equality (\ref{eq:GkPm}) follows from the observations above.
\end{proof}

\subsection{Heaps of type \texorpdfstring{$II$}{II} for \texorpdfstring{$(1,k)$}{(1,k)}-Dyck paths 
and generating functions}
\label{sec:GDGF}
In this section, we focus on the heaps of type $II$ for $(1,k)$-Dyck paths with 
$k\ge1$.
We calculate the generating functions of heaps of type $II$ associated to 
the generalized  $(1,k)$-Dyck paths by use of the inversion lemma 
(see Proposition \ref{prop:inversion}).

A heap $H$ of type $II$ for $(1,k)$-Dyck paths of size $n$ satisfies the following simple three conditions:
\begin{enumerate}[(H1)]
\item The size of a piece is $k$.
\item A heap $H$ has a unique maximal piece and its left abscissa 
is one. 
\item The number of pieces is $n+1$.
\end{enumerate}

In Figure \ref{fig:k2}, we list up all heaps with three pieces and $(n,k)=(2,2)$.
Note that the number of pieces is three which is one plus $n=2$.

\begin{figure}[ht]
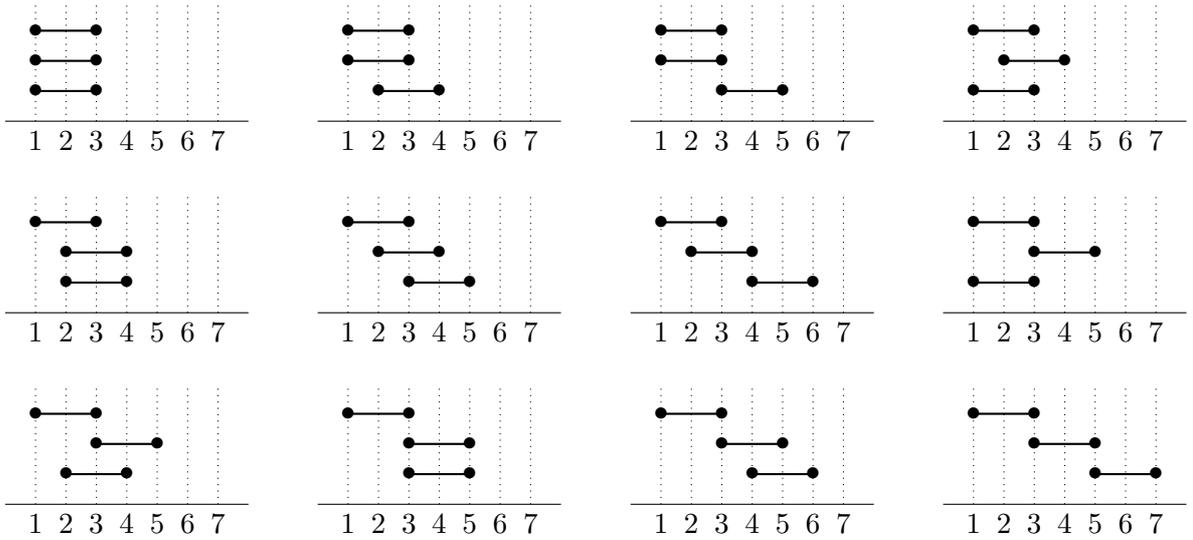

\tikzpic{-0.5}{[scale=0.4]
\draw(0,0)--(8,0);
\foreach \x in {1,...,7}{
\draw[dotted](\x,0)node[anchor=north]{$\x$}--(\x,4);
}
\draw[thick](1,3)node{$\bullet$}--(3,3)node{$\bullet$}(1,2)node{$\bullet$}--(3,2)node{$\bullet$}
(1,1)node{$\bullet$}--(3,1)node{$\bullet$};
}
\quad
\tikzpic{-0.5}{[scale=0.4]
\draw(0,0)--(8,0);
\foreach \x in {1,...,7}{
\draw[dotted](\x,0)node[anchor=north]{$\x$}--(\x,4);
}
\draw[thick](1,3)node{$\bullet$}--(3,3)node{$\bullet$}(1,2)node{$\bullet$}--(3,2)node{$\bullet$}
(2,1)node{$\bullet$}--(4,1)node{$\bullet$};
}
\quad
\tikzpic{-0.5}{[scale=0.4]
\draw(0,0)--(8,0);
\foreach \x in {1,...,7}{
\draw[dotted](\x,0)node[anchor=north]{$\x$}--(\x,4);
}
\draw[thick](1,3)node{$\bullet$}--(3,3)node{$\bullet$}(1,2)node{$\bullet$}--(3,2)node{$\bullet$}
(3,1)node{$\bullet$}--(5,1)node{$\bullet$};
}
\quad
\tikzpic{-0.5}{[scale=0.4]
\draw(0,0)--(8,0);
\foreach \x in {1,...,7}{
\draw[dotted](\x,0)node[anchor=north]{$\x$}--(\x,4);
}
\draw[thick](1,3)node{$\bullet$}--(3,3)node{$\bullet$}(2,2)node{$\bullet$}--(4,2)node{$\bullet$}
(1,1)node{$\bullet$}--(3,1)node{$\bullet$};
}
\\[11pt]
\tikzpic{-0.5}{[scale=0.4]
\draw(0,0)--(8,0);
\foreach \x in {1,...,7}{
\draw[dotted](\x,0)node[anchor=north]{$\x$}--(\x,4);
}
\draw[thick](1,3)node{$\bullet$}--(3,3)node{$\bullet$}(2,2)node{$\bullet$}--(4,2)node{$\bullet$}
(2,1)node{$\bullet$}--(4,1)node{$\bullet$};
}
\quad
\tikzpic{-0.5}{[scale=0.4]
\draw(0,0)--(8,0);
\foreach \x in {1,...,7}{
\draw[dotted](\x,0)node[anchor=north]{$\x$}--(\x,4);
}
\draw[thick](1,3)node{$\bullet$}--(3,3)node{$\bullet$}(2,2)node{$\bullet$}--(4,2)node{$\bullet$}
(3,1)node{$\bullet$}--(5,1)node{$\bullet$};
}
\quad
\tikzpic{-0.5}{[scale=0.4]
\draw(0,0)--(8,0);
\foreach \x in {1,...,7}{
\draw[dotted](\x,0)node[anchor=north]{$\x$}--(\x,4);
}
\draw[thick](1,3)node{$\bullet$}--(3,3)node{$\bullet$}(2,2)node{$\bullet$}--(4,2)node{$\bullet$}
(4,1)node{$\bullet$}--(6,1)node{$\bullet$};
}
\quad
\tikzpic{-0.5}{[scale=0.4]
\draw(0,0)--(8,0);
\foreach \x in {1,...,7}{
\draw[dotted](\x,0)node[anchor=north]{$\x$}--(\x,4);
}
\draw[thick](1,3)node{$\bullet$}--(3,3)node{$\bullet$}(3,2)node{$\bullet$}--(5,2)node{$\bullet$}
(1,1)node{$\bullet$}--(3,1)node{$\bullet$};
}
\\[11pt]
\tikzpic{-0.5}{[scale=0.4]
\draw(0,0)--(8,0);
\foreach \x in {1,...,7}{
\draw[dotted](\x,0)node[anchor=north]{$\x$}--(\x,4);
}
\draw[thick](1,3)node{$\bullet$}--(3,3)node{$\bullet$}(3,2)node{$\bullet$}--(5,2)node{$\bullet$}
(2,1)node{$\bullet$}--(4,1)node{$\bullet$};
}
\quad
\tikzpic{-0.5}{[scale=0.4]
\draw(0,0)--(8,0);
\foreach \x in {1,...,7}{
\draw[dotted](\x,0)node[anchor=north]{$\x$}--(\x,4);
}
\draw[thick](1,3)node{$\bullet$}--(3,3)node{$\bullet$}(3,2)node{$\bullet$}--(5,2)node{$\bullet$}
(3,1)node{$\bullet$}--(5,1)node{$\bullet$};
}
\quad
\tikzpic{-0.5}{[scale=0.4]
\draw(0,0)--(8,0);
\foreach \x in {1,...,7}{
\draw[dotted](\x,0)node[anchor=north]{$\x$}--(\x,4);
}
\draw[thick](1,3)node{$\bullet$}--(3,3)node{$\bullet$}(3,2)node{$\bullet$}--(5,2)node{$\bullet$}
(4,1)node{$\bullet$}--(6,1)node{$\bullet$};
}
\quad
\tikzpic{-0.5}{[scale=0.4]
\draw(0,0)--(8,0);
\foreach \x in {1,...,7}{
\draw[dotted](\x,0)node[anchor=north]{$\x$}--(\x,4);
}
\draw[thick](1,3)node{$\bullet$}--(3,3)node{$\bullet$}(3,2)node{$\bullet$}--(5,2)node{$\bullet$}
(5,1)node{$\bullet$}--(7,1)node{$\bullet$};
}	
\caption{Twelve heaps with three pieces for $k=2$.}
\label{fig:k2}
\end{figure}

The valuation $v(H)$ is given by 
\begin{align*}
v(H):=x^{\text{number of pieces}}p^{\sum(\text{left abscissae of pieces})}.
\end{align*} 
Then, the generating function $G^{(k)}(x,p)$ is defined to be
\begin{align}
\label{eq:Gkxp}
G^{(k)}(x,p):=\sum_{H}v(H),
\end{align}	
where the sum is taken over all heaps $H$ (of arbitrary size) of type $II$.

We calculate $G^{(k)}(x,p)$ by use of Proposition \ref{prop:inversion}.
First, we calculate the generating function of trivial heaps.
We define $T^{(k)}(x,p):=\sum_{n\ge0}(-1)^{n}t_{n}x^{n}$ as the ordinary generating function of 
trivial heaps.
The formal power series $T^{(k)}(x,p)$ satisfies 
\begin{align*}
T^{(k)}(x,p)&=1+(-1)xp T^{(k)}(p^{k+1}x,p)+(-1)xp^2 T^{(k)}(p^{k+2}x,p)+\cdots, \\
&=1+(-1)x\sum_{j\ge1}p^{j}T^{(k)}(p^{k+j}x,p). 
\end{align*}
Each term in the above equation corresponds to a heap such that we have a piece $p$ at position $j$ and 
there is no piece left to $p$.
This piece gives the factor $xp^j$ and $T^{(k)}(p^{k+j}x,p)$ gives the generating function of 
heaps right to $p$.
The above functional relation implies 
\begin{align*}
t_{n}&=\genfrac{}{}{}{}{p^{(k+1)(n-1)+1}}{1-p^{n}}t_{n-1}, \\
&=\genfrac{}{}{}{}{p^{n+(k+1)n(n-1)/2}}{(p;p)_{n}},
\end{align*}
where we have used the initial condition $t_{0}=1$, and $(a;q)_{n}=\prod_{k=0}^{n-1}(1-aq^{k})$ 
is the $q$-Pochhammer symbol.

By applying Proposition \ref{prop:inversion} to our model, 
the generating function $G^{(k)}(x,p)$ is given by 
\begin{align*}
G^{(k)}(x,p)&=\genfrac{}{}{}{}{T^{(k)}(px,p)}{T^{(k)}(x,p)}, \\
&=\genfrac{}{}{}{}{\displaystyle\sum_{n\ge0}(-1)^{n}x^{n}p^{2n+(k+1)n(n-1)/2}(p;p)_{n}^{-1}}
{\displaystyle\sum_{n\ge0}(-1)^{n}x^{n}p^{n+(k+1)n(n-1)/2}(p;p)_{n}^{-1}}.
\end{align*}

\begin{prop}
The generating function $G^{(k)}(x,p)$ satisfies the 
following functional relation:
\begin{align}
\label{eq:Gkfe}
G^{(k)}(x,p)=1+px \prod_{j=0}^{k}G^{(k)}(p^{j}x,p).
\end{align}
\end{prop}
\begin{proof}
We derive the functional relation for $T^{(k)}(x,p)$. 
We have 
\begin{align*}
T^{(k)}(px,p)&=\sum_{n\ge0}(-1)^{n}x^{n}\genfrac{}{}{}{}{p^{2n+(k+1)n(n-1)/2}}{(p;p)_{n}}, \\
&=T^{(k)}(x,p)-\sum_{n\ge1}(-1)^{n}x^{n}\genfrac{}{}{}{}{p^{n+(k+1)n(n-1)/2}}{(p;p)_{n-1}}, \\
&=T^{(k)}(x,p)-\sum_{n\ge0}(-1)^{n+1}x^{n+1}\genfrac{}{}{}{}{p^{n+1+(k+1)n(n+1)/2}}{(p;p)_{n}}, \\
&=T^{(k)}(x,p)+xp T^{(k)}(p^{k+1}x,p).
\end{align*}
By dividing both sides by $T^{(k)}(x,p)$, we obtain Eq. (\ref{eq:Gkfe}).
\end{proof}

Let $\mu_{0}(n):=(UD^{k})^{n}$ and $\mu$ be a generalized Dyck path below $\mu_{0}$. 
Then, we define a statistic $\mathrm{Area}(\mu)$ be the number of unit boxes 
above $\mu$ and below $\mu_{0}(n)$.
\begin{prop}
\label{prop:GinArea}
The generating function $G^{(k)}(x,p)$ is expressed in terms of generalized Dyck paths by
\begin{align}
\label{eq:GinArea}
G^{(k)}(x,p)=\sum_{n\ge0}(xp)^{n}\sum_{\mu\le \mu_{0}(n)}p^{\mathrm{Area}(\mu)}.
\end{align}
\end{prop}
\begin{proof}
Equation (\ref{eq:GinArea}) follows from the straightforward calculation 
(see for example \cite{Shi22b}).
We give a sketch of a proof.
A $(1,k)$-Dyck path can be expressed as 
\begin{align*}
U\circ\mu_{1}\circ D\circ \mu_2\circ D\circ\ldots \circ \mu_{k}\circ D\circ \mu_{k+1},
\end{align*}
where $\mu_i$, $1\le i\le k+1$, are $(1,k)$-Dyck paths.
By rewriting this decomposition in terms of the generating function $G^{(k)}(x,p)$, 
one can show that Eq. (\ref{eq:GinArea}) satisfies Eq. (\ref{eq:Gkfe}).
\end{proof}

Given a $(1,k)$-Dyck path $\mu$, we have a bijection between a heap $H$ of type $II$
and the path $\mu$ via $\mathfrak{a}(\mu_{0})$ and $\mathbf{c}$ defined 
in Section \ref{subsec:HGFII}.
This bijection preserves the valuation given to $H$ and $\mu$, {\it i.e.}, 
the sum of the left abscissae of pieces in $H$ is equal to the 
number $\mathrm{Area}(\mu)$ of unit boxes above $\mu$ and below $\mu_{0}$.

\subsection{Generalized Dyck paths and heaps in \texorpdfstring{$k+1$}{k+1}-dimensions}
We introduce heaps in $k+1$-dimensions for $(1,k)$-Dyck paths.
Let $\mathcal{I}(k)$ be the set of points in $k$-dimensions 
\begin{align*}
\mathcal{I}(k)
:=
\left\{\mathbf{x}=(x_1,x_2,\ldots, x_{k}) \bigg| 
\begin{array}{c}x_i\ge1, \\
\mathbf{x} \text{ is weakly decreasing} \\
x_{1}-x_{k}\in\{0,1\}
\end{array}   \right\}.		
\end{align*}
A piece of our model is a unit $k$-dimensional hypercube.
Let $(h_1,h_2,\ldots,h_{k+1})$ be the coordinate of the center of a hypercube.
We pile hypercubes as pieces of a heap such that 
\begin{enumerate}[({I}1)]
\item The coordinates satisfies $(h_1,h_2,\ldots,h_{k})\in\mathcal{I}(k)$.
\item $h_{k+1}\in\mathbb{N}_{\ge0}$.
\item A heap has a unique maximal piece such that its projection to the hyper plane $h_{k+1}=0$
is $(1,1,\ldots,1)$.	
\end{enumerate}

\begin{defn}
We denote by $\mathcal{H}(\mathcal{I}(k))$ the set of heaps 
satisfying the conditions (I1), (I2) and (I3).
\end{defn}

\begin{remark}
Note that $\mathcal{I}(k)$ is the subset of points in $\mathbb{N}_{\ge1}^{k}$.
A heap in $\mathcal{H}(\mathcal{I}(k))$ can be visualized in $k+1$-dimensions.
\end{remark}

Below, we study heaps in $\mathcal{H}(\mathcal{I}(k))$ in terms of a $(1,k)$-Dyck path 
and a $k$-tuple Dyck paths. 
First, we introduce a correspondence between a $(1,k)$-Dyck path and a $k$-tuple of 
Dyck paths.
Then, we construct a bijection between a heap in $\mathcal{H}(\mathcal{I}(k))$ 
and a $k$-tuple Dyck paths.
Finally, we show that the generating function of heaps in $\mathcal{H}(\mathcal{I}(k))$
coincides with the generating function of heaps for $(1,k)$-Dyck paths introduced in 
Section \ref{sec:GDGF}.

Let $\mathcal{D}$ be a $(1,k)$-Dyck path of size $n$.
We decompose the path $\mathcal{D}$ into a $k$-tuple Dyck paths following \cite{Shi21b}.
Recall $\mathcal{D}$ is a lattice path from $(0,0)$ to $(kn,n)$ which is below $y=x/k+1$ and above the 
line $y=0$.
Let $\mathfrak{m}(\mathcal{D}):=(m_1,\ldots,m_{n+1})$ be a sequence of non-negative 
integers such that 
\begin{enumerate}
\item $m_{n+1}=0$.
\item Let $U_{i}$ be an up step in $\mathcal{D}$ such that it connect the point 
$(\alpha,i-1)$ with $(\alpha,i)$ for some $\alpha$.
The number $m_{i}$, $1\le i\le n$, is the number of unit boxes which is right to 
$U_{i}$ and left to the line $x=kn$.
\end{enumerate}
We call a sequence $\mathfrak{m}(\mathcal{D})$ a {\it step sequence} for $\mathcal{D}$.
Note that we have a unique step sequence for a generalized Dyck path if $k$ given.
This implies that we have a natural bijection between $\mathcal{D}$ and $\mathfrak{m}(\mathcal{D})$.

We construct a $k$-tuple Dyck paths from $\mathfrak{m}(\mathcal{D})$:
\begin{defn}[Definition 5.6 in \cite{Shi21b}]
A $k$-tuple of step sequences $\{\mathfrak{m}(D_{i})\}_{i=1}^{k}$ is recursively defined as follows.
\begin{enumerate}
\item Set $i=k$ and $\mathfrak{m}:=\mathfrak{m}(\mathcal{D})$.
\item We define 
\begin{align}
\label{eq:mceil}
\mathfrak{m}(D_{i}):=\left\lceil\genfrac{}{}{}{}{\mathfrak{m}}{i} \right\rceil.
\end{align}
\item Replace $\mathfrak{m}$ with $\mathfrak{m}-\mathfrak{m}(D_{i})$.
Decrease $i$ by one, and go to $(2)$. The algorithm stops when $i=1$.
\end{enumerate}
\end{defn}

\begin{remark}
We have a $k$-tuple of step sequence $\{\mathfrak{m}(D_i)\}_{i=1}^{k}$.
From Eq. (\ref{eq:mceil}), $\mathfrak{m}(D_{i})$ can be regarded as 
a step sequence of a Dyck path $D_{i}$.
Then, the Dyck path $D_{i}$ is above the Dyck path $D_{i+1}$ for 
all $1\le i\le k-1$.  
\end{remark}

\begin{lemma}
\label{lemma:mDi}
The sequence $\mathfrak{m}(D_{i})$ is weakly decreasing. 
\end{lemma}
\begin{proof}
By construction, the step sequence $\mathfrak{\mathcal{D}}$ is weakly decreasing.
The equation (\ref{eq:mceil}) insures that $\mathfrak{m}(D_{i})$ is 
weakly decreasing.
\end{proof}

By definition, the length of $\mathfrak{m}(D_{i})$ is $n$. We define $M_{i,j}$ 
such that $\mathfrak{m}(D_{i})=:(M_{i,1},M_{i,2},\ldots,M_{i,n+1})$ for $1\le i\le k$.

\begin{lemma}
\label{lemma:Mij}
We have $\{M_{i,j}:1\le i\le k\}$ is weakly decreasing and 
$\max\{M_{i,j}:1\le i\le k\}-\min\{M_{i,j}:1\le i\le k\}\in\{0,1\}$.
\end{lemma}
\begin{proof}
Suppose that $m_{i}=km+k'$ with $m\ge0$ and $k'\in[0,k-1]$.
From Eq. (\ref{eq:mceil}), we have 
$\max\{M_{i,j}:1\le i\le k\}-\min\{M_{i,j}:1\le i\le k\}=0$ 
if $k'=0$. 
If $k'\neq 0$, we have $M_{i,j}=m$ for $i\ge k'+1$ and $M_{i,j}=m+1$ 
for $i\le k'$.
This implies $\{M_{i,j}:1\le i\le k\}$ is weakly decreasing and 
$\max\{M_{i,j}:1\le i\le k\}-\min\{M_{i,j}:1\le i\le k\}=1$,
which completes the proof.
\end{proof}

Then, we define 
\begin{align}
\label{eq:xij}
x_{i,j}:=j-M_{k+1-i,n+2-j},
\end{align}
for $1\le i\le k+1$ and $1\le j\le n+1$, and 
\begin{align}
\label{eq:xi}
\mathbf{x}_{i}:=(x_{1,i},x_{2,i},\ldots,x_{k,i}),
\end{align} 
for $1\le i\le n+1$.

Since $m_{n+1}=0$ in $\mathfrak{m}(\mathcal{D})$, we have 
$\mathbf{x}_{1}=(1,1,\ldots,1)$.
\begin{lemma}
\label{lemma:sjtj}
Let $\mathbf{x}_{i}:=(s_1,s_2,\ldots,s_{k})$ and  $\mathbf{x}_{i+1}=(t_1,t_2,\ldots,t_{k})$.
We have 
\begin{enumerate}
\item $-1\le s_{j}-t_{j}\le i-1$ for $1\le j\le k$.
\item $s_{j}$ is weakly decreasing and $\max(\mathbf{x}_{i})\le \min(\mathbf{x}_{i})+1$.
\end{enumerate}
\end{lemma}
\begin{proof}
(1). 
From the definition of $\mathbf{x}_{i}$, we have 
\begin{align*}
s_{j}-t_{j}&=x_{j,i}-x_{j,i+1}, \\
&=-1+M_{n+2-j,n+1-i}-M_{n+2-j,n+2-i},
\end{align*}
Note that we have $M_{j,i}\ge M_{j,i+1}$ from Lemma \ref{lemma:mDi}, 
and $M_{j,i}\le n+1-i$ since $\mathfrak{m}(D_{i})$ is a step sequence 
of a Dyck path.
From these observations, we have $-1\le s_{j}-t_{j}\le i-1$.

(2). From Lemma \ref{lemma:Mij} and Eq. (\ref{eq:xi}), it is obvious 
that $\{s_{1},\ldots,s_{k}\}$ is weakly decreasing and $\max(\mathbf{x}_{i})\le\min(\mathbf{x}_{i})+1$.
\end{proof}

Let $\mathbf{x}=(x_1,\ldots,x_{k})$ and $\mathbf{y}=(y_1,\ldots,y_{k})$ 
be elements in $\mathcal{I}(k)$.
We consider a graph $\mathcal{G}(k)$ whose vertex set is $\mathcal{I}(k)$ and the set $E(k)$ of edges 
such that 
\begin{align*}
(\mathbf{x},\mathbf{y})\in E(k) \Leftrightarrow x_{i}-y_{i}\in\{-1,0,1\} \text{ for } 1\le i\le k.
\end{align*}
Note that $(\mathbf{x},\mathbf{x})\in E(k)$ for all $\mathbf{x}\in\mathcal{I}(k)$.

Graphically, $\mathcal{G}(3)$ partially looks as in Figure \ref{fig:kdim}.
We omit loops from $\mathbf{x}$ to $\mathbf{x}$, and edges which connect $\mathbf{x}$ with 
$\mathbf{y}$ such that $x_1=y_{1}=3$.
We identify the vertices with the same label.
\begin{figure}[ht]
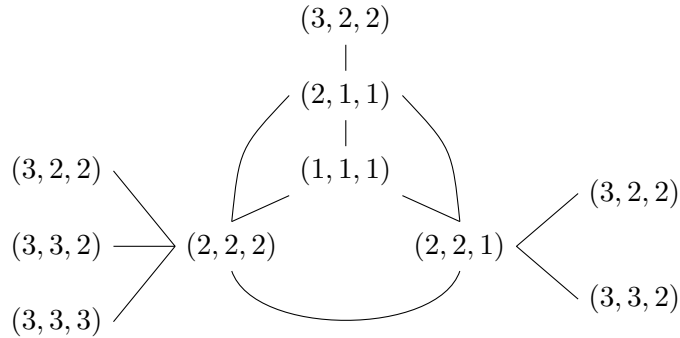

\tikzpic{-0.5}{
\node(0)at(0,0){$(1,1,1)$};
\node(1)at(0,1){$(2,1,1)$};
\node(2)at(0,2){$(3,2,2)$};
\draw(0.north)--(1.south)(1.north)--(2.south);
\node(3)at(1.5,-1){$(2,2,1)$};
\node(4)at(3.8,-0.3){$(3,2,2)$};
\node(5)at(3.8,-1.7){$(3,3,2)$};
\draw(0.{south east})--(3.north)(3.east)--(4.west)(3.east)--(5.west);
\node(6)at(-1.5,-1){$(2,2,2)$};
\node(7)at(-3.8,0){$(3,2,2)$};
\node(8)at(-3.8,-1){$(3,3,2)$};
\node(9)at(-3.8,-2){$(3,3,3)$};
\draw(0.south west)--(6.north)(6.west)--(7.east)(6.west)--(8.east)(6.west)--(9.east);
\draw(1.east)..controls(1.4,0.3)..(3.north)(3.south)..controls(1.2,-2.2)and(-1.2,-2.2)..(6.south)(6.north)..controls(-1.4,0.3)..(1.west);
}
\caption{A graph $\mathcal{G}(3)$ without loops and some edges}
\label{fig:kdim}
\end{figure}

Suppose we have an $n+1$-tuple $\mathcal{X}:=\{\mathbf{x}_{i}\}_{i=1}^{n+1}$ defined 
by Eq. (\ref{eq:xi}).
Given $i$, let $j<i$ be the integer such that 
$(\mathbf{x}_{i},\mathbf{x}_{j'})\notin E(k)$ for $j+1\le j'\le i-1$ and 
$(\mathbf{x}_{i},\mathbf{x}_{j})\in E(k)$.
We write this relation as $\mathbf{x}_{i}\rightarrow\mathbf{x}_{j}$.
This relation defines a directed graph on the set $\mathcal{X}$.
We say that $\mathbf{x}_{i}$ and $\mathbf{x}_{j}$ are connected 
if there is an undirected path from $\mathbf{x}_{i}$ to $\mathbf{x}_{j}$ 
on the graph.

\begin{defn}
We define a forest $\mathtt{For}(\mathcal{X})$ from $\mathcal{X}$ as 
\begin{align*}
\mathtt{For}(\mathcal{X}):=\{T_{i}: 1\le i\le r\},
\end{align*}
where $T_{i}$ is a connected component on the directed graph of $\mathcal{X}$.
The integer $r$ is the number of connected components.
\end{defn}

One can construct a forest from an $n+1$-tuple $\mathcal{X}$.
The next lemma follows from the properties of $\mathcal{X}$.

\begin{lemma}
\label{lemma:ForX}
Suppose $\mathcal{X}$ is an $n+1$-tuple constructed by Eq. (\ref{eq:xi}).
Then, $\mathtt{For}(\mathcal{X})$ is a planar rooted tree, i.e.,
it satisfies 
\begin{enumerate}
\item $\mathtt{For}(\mathcal{X})$ consists of a unique connected component.
\item There is no outgoing edge from $\mathbf{x}_{1}$, which means that 
$\mathbf{x}_{1}$ is the root of a tree.
\item There exists a unique directed path from $\mathbf{x}_{i}$ to 
the root $\mathbf{x}_{1}$ for all $2\le i\le n+1$.
\end{enumerate}
\end{lemma}
\begin{proof}
By definition, (2) is obvious. 
The condition (1) and the definition of $\mathbf{x}_{i}\rightarrow\mathbf{x}_{j}$ 
implies the condition (3).
It is enough to prove the condition (1).

We show (1) by induction on $p\in[1,n+1]$.
Recall that we have $\mathbf{x}_{1}=(1,1,\ldots,1)$.
From Lemma \ref{lemma:sjtj}, we have $\mathbf{x}_{2}\rightarrow\mathbf{x}_{1}$.
Suppose that the set $\{\mathbf{x}_{i}\}_{i=1}^{p}$ forms a planar 
rooted tree satisfying the three conditions.
Then, we will show that $\mathbf{x}_{p+1}\rightarrow\mathbf{x}_{q}$ 
for some $q\in[1,p]$.
Let $\mathbf{x}_{p}:=(s_1,\ldots,s_{k})$ and $\mathbf{x}_{p+1}:=(t_1,\ldots,t_{k})$.
We consider two cases: a) $s_{j}-t_{j}=-1$ for some $j\in[1,k]$, and 
b) $s_{j}-t_{j}\ge 0$ for all $j\in[1,k]$.

\paragraph{Case a)}
Suppose that $s_{j}=t_{j}-1$ for some $j\in[1,k]$.
Let $s:=\max(\mathbf{x}_{i})$ and $s':=\min(\mathbf{x}_{i})$. From Lemma \ref{lemma:sjtj},
we have $s'\in\{s,s-1\}$.
We have two cases: i) $s_{j}=s$, and ii) $s_{j}=s'$.
\begin{enumerate}[i)]
\item Let
From Lemma \ref{lemma:sjtj}, we have $s_{i}=s$ for $1\le i\le j$ and this implies 
that $t_{i}=s+1$ for $1\le i\le j$.
Suppose that $s_{i'}=s$ for $j+1\le i'\le k'$ and $s_{i'}=s-1$ for $k'+1\le i'\le k$.
Since $t_{i'}\le s_{i'}+1$, $t_{i'}\le s$ for $k'+1\le i'\le k$. From Lemma \ref{lemma:sjtj},
we have $t_{i'}=s$ for $k'+1\le i'\le k$.
From Lemma \ref{lemma:sjtj}, $s_{i'}\in[s+1,s]$ and weakly decreasing.
From these, we have $t_{j}-s_{j}\in[0,1]$ for all $j\in[1,k]$.
By definition of $E(k)$, we have $\mathbf{x}_{p+1}\rightarrow\mathbf{x}_{p}$. 
\item
Suppose that $s_{i}=s'+1$ for $1\le i\le k'$ and $s_{i}=s'$ for $k'+1\le i\le k$.
We have $k'<j$ since $s_{j}=s'$.
From Lemma \ref{lemma:sjtj}, we have $t_{i}=s'+1$ for $k'+1\le i\le j$ and 
$t_{i}\in\{s'+1,s'+2\}$ or $t_{i}\in\{s',s'+1\}$ for $i\in[1,k]\setminus[k'+1,j]$.
From (2) in Lemma \ref{lemma:sjtj}, we have $t_{i}-s_{i}\in\{0,1\}$, 
which implies $\mathbf{x}_{p+1}\rightarrow\mathbf{x}_{p}$.
\end{enumerate}

\paragraph{Case b)}
If all $s_{j}-t_{j}\in\{0,1\}$, we have $\mathbf{x}_{p+1}\rightarrow\mathbf{x}_{p}$.
Suppose that there exist $j'$ such that $s_{j'}-t_{j'}\ge2$.
Then, $(\mathbf{x}_{p+1},\mathbf{x}_{p})\notin E(k)$.
Since $s_{j}-t_{j}\ge0$ for all $j\in[1,k]$, the $p$-tuple 
$\mathcal{X}':=(\mathbf{x}_{1},\ldots,\mathbf{x}_{p-1},\mathbf{x}_{p+1})$ 
gives a $p$-tuple of step sequences. 
By induction, the forest $\mathtt{For}(\mathcal{X}')$ is a planar rooted 
tree. This means that $\mathbf{x}_{p+1}\rightarrow\mathbf{x}_{p'}$ 
with some $p'\le p-1$.
From these observations, we have $\mathbf{x}_{p+1}\rightarrow\mathbf{x}_{q}$ 
for some $q\le p$. This completes the proof.
\end{proof}

\begin{prop}
\label{prop:bijprt}
A heap with $n+1$ pieces in $\mathcal{H}(\mathcal{I}(k))$ with conditions (I1), (I2) and (I3) is 
bijective to a planar rooted tree $\mathtt{For}(\mathcal{X})$.
\end{prop}
\begin{proof}
Let $\mathcal{X}:=(\mathbf{x}_1,\ldots,\mathbf{x}_{n+1})$.
From (2) in Lemma \ref{lemma:sjtj}, we have $\mathbf{x}_{i}\in\mathcal{I}(k)$. Therefore, we have (I1).
From Lemma \ref{lemma:ForX}, the forest $\mathtt{For}(\mathcal{X})$ is a planar rooted tree.
Let $h_{k+1}$ (resp. $h'_{k+1}$) be the $k+1$-th coordinate of $\mathbf{x}_{j}$ (resp. $\mathbf{x}_{i}$).
The condition $\mathbf{x}_{j}\rightarrow\mathbf{x}_{i}$ implies that two pieces $\mathbf{x}_{j}$ 
and $\mathbf{x}_{i}$ shares at least one vertex in $\mathcal{I}(k)$ and $h_{k+1}<h'_{k}$.
This gives the order of pieces $\mathbf{x}_{i}$.
Recall that $\mathtt{For}(\mathcal{X})$ has a unique vertex, the root $\mathbf{x}_{1}$.
Since there is no outgoing edges from $\mathbf{x}_{1}$, there is no piece above $\mathbf{x}_{1}$ 
in the heap. This means that $\mathbf{x}_{1}=(1,\ldots,1)$ is the maximal piece. The condition (I3)
follows. 
We define the $k+1$-th coordinate $h_{k+1}=0$ for the piece $\mathbf{x}_{p}$ 
which has no incoming edges in $\mathtt{For}(\mathcal{X})$.
The $k+1$-th coordinates of other pieces are larger than zero. The condition (I2) follows.
From these observations, $\mathtt{For}(\mathcal{X})$ is bijective to 
a heap on $\mathcal{I}(k)$ with conditions (I1), (I2) and (I3).
\end{proof}

\begin{example}
We have $22$ planar rooted trees for $(k,n)=(3,2)$. Below, we list up all $22$ trees, where
the root vertex $\mathbf{x}_{1}=(1,1,1)$ is omitted.
From Figure \ref{fig:kdim}, we have 
\begin{align*}
&(1,1,1)\rightarrow(1,1,1),\quad (1,1,1)\rightarrow(2,1,1),\quad (1,1,1)\rightarrow(2,2,1),\quad
(1,1,1)\rightarrow(2,2,2), \\
&(2,1,1)\rightarrow(2,1,1),\quad (2,1,1)\rightarrow(1,1,1),\quad
(2,2,1)\rightarrow(1,1,1),\quad (2,2,1)\rightarrow(2,2,1), \\
&(2,2,2)\rightarrow(1,1,1),\quad (2,2,2)\rightarrow(2,2,2),\quad (2,1,1)\rightarrow(2,2,1),
\quad (2,1,1)\rightarrow(2,2,2), \\
&(2,2,1)\rightarrow(2,1,1),\quad (2,2,1)\rightarrow(2,2,2),\quad (2,2,2)\rightarrow(2,1,1),\quad 
(2,2,2)\rightarrow(2,2,1),\\ 
&(3,2,2)\rightarrow(2,1,1),\quad (3,2,2)\rightarrow(2,2,1),\quad (3,3,2)\rightarrow(2,2,1),\quad
(3,2,2)\rightarrow(2,2,2),\\
&(3,3,2)\rightarrow(2,2,2),\quad (3,3,3)\rightarrow(2,2,2).
\end{align*}
In the case of $(n,k)=(3,2)$, all trees are of the form 
$\mathbf{x}_{3}\rightarrow\mathbf{x}_2\rightarrow\mathbf{x}_1$.

Consider a step sequence $\mathfrak{m}:=(8,1,0,0)$ for $(n,k)=(3,3)$. 
Then, we have 
\begin{align*}
\mathbf{x}_1=(1,1,1),\quad \mathbf{x}_2=(2,2,2),\quad \mathbf{x}_3=(3,3,2),\quad \mathbf{x}_4=(2,1,1).
\end{align*}
The planar plane tree is given by 
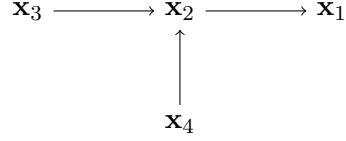
\begin{align*}
\tikzpic{-0.5}{
\node(0)at (0,0){$\mathbf{x}_3$};
\node(1)at (2,0){$\mathbf{x}_2$};
\node(2)at (4,0){$\mathbf{x}_1$};
\node(3)at (2,-1.5){$\mathbf{x}_4$};
\draw[->](0.east)--(1.west);
\draw[->](1.east)--(2.west);
\draw[->](3.north)--(1.south);
}
\end{align*}
Note that $(\mathbf{x}_3,\mathbf{x}_4)\notin E(k)$.
\end{example}

Let $H_{i}$, $1\le i\le k$, be a heap of type $II$ corresponding to
a Dyck path $D_{i}$. 
By definition, a heap $H_{i}$ contains $n$ dimers.
We can introduce a linear order on dimers by removing a right-most 
maximal piece one-by-one.

\begin{prop}
\label{prop:xipj}
Let $p^{j}_i$, $1\le i\le n$, be a $i$-th dimer from top in the linear order in
the heap $H_{j}$, $1\le j\le k$.
We denote by $L(p^{j}_i)$ the left abscissa of the dimer $p^{j}_i$.
We have
\begin{align}
\label{eq:xipj}
\mathbf{x}_{i}=(L(p^{1}_{i}),L(p^2_{i}),\ldots,L(p^{k}_{i})),
\end{align}
where $\mathbf{x}_i$ is defined by Eq. (\ref{eq:xi}).
\end{prop}
\begin{proof}
Equation (\ref{eq:mceil}) gives a step sequence $\mathbf{m}(D_{j})$ for a 
Dyck paths $D_{j}$.
By the correspondence between a lattice path and a heap of type $II$ 
(see Figure \ref{fig:corLGV}),
a step sequence gives the position of pieces from the left end.
By combining these observations with the definition of $x_{i,j}$ in 
Eq. (\ref{eq:xij}), Eq. (\ref{eq:xipj}) follows.
\end{proof}

The next corollary is a direct consequence of Propositions \ref{prop:bijprt} and \ref{prop:xipj}.
\begin{cor}
\label{cor:bij}
A heap in $\mathcal{H}(\mathcal{I}(k))$ is bijective to the following combinatorial objects.
\begin{enumerate}
\item A $(1,k)$-Dyck path $\mu$, equivalently a $k$-tuple of Dyck paths $\{D_1,\ldots, D_{k}\}$.
\item A heap of type $II$ for the path $\mu$.
\item A $k$-tuple of heaps of type $II$ for the Dyck paths $\{D_1,\ldots,D_{k}\}$.
\item A planar rooted tree $\mathtt{For}(\mathcal{X})$.
\end{enumerate}
\end{cor}

Corollary \ref{cor:bij} gives a bijection among combinatorial objects.
We also have a correspondence at the level of generating functions.

Let $\mathcal{P}(k;n)$ be the set of planar rooted trees consisting of $\{\mathbf{x}_i\}_{i=1}^{n+1}$.
The valuation of a planar rooted tree  $T$ is defined to be
\begin{align*}
v(T):=\sum_{i,j}x_{i,j}-(n+1)k,
\end{align*}
where $x_{i,j}$ is given by Eq. (\ref{eq:xij}).
We define the generating function $\mathcal{F}(k;x,p)$ by 
\begin{align*}
\mathcal{F}(k;x,p):=\sum_{n\ge0}x^{n}\sum_{T\in\mathcal{P}(k;n)}p^{v(T)}.
\end{align*}

\begin{prop}
We have $\mathcal{F}(k;x,p)=G^{(k)}(p^{-1}x,p)$ where the generating function 
$G^{(k)}(x,p)$ is given by Eq. (\ref{eq:Gkxp}).
\end{prop}
\begin{proof}
Let $P$ be a $(1,k)$-Dyck path of size $n$ and $\mathbf{m}(P):=(m_1,\ldots,m_{n+1})$ be 
the step sequence of $P$. 
We have $k$ Dyck paths $D_{i}$, $1\le i\le k$, by definition (\ref{eq:mceil}).
We denote by $|\mathbf{m}(P)|$ the sum of the entries of the step sequence for $P$.
By definitions (\ref{eq:mceil}) of $\mathbf{m}(D_{i})$, we have 
\begin{align*}
\sum_{i=1}^{k}|\mathbf{m}(D_{i})|=\genfrac{}{}{}{}{(n+1)(n+2)k}{2}-\sum_{i,j}x_{i,j},
\end{align*} 
where $x_{i,j}$ is defined by Eq. (\ref{eq:xij}) via the step sequences of $D_{i}$, $1\le i\le k$.
Then, by definition of the valuation of a planar tree we have
\begin{align*}
\sum_{i=1}^{k}|\mathbf{m}(D_{i})|=\genfrac{}{}{}{}{n(n+1)k}{2}-v(T),
\end{align*}
where $T$ is a planar rooted tree constructed from $D_{i}$, $1\le i\le k$.
Note that the number $n(n+1)k/2$ is the number of unit boxes between the lowest $(1,k)$-Dyck path 
and the highest $(1,k)$-Dyck path.
This implies that 
\begin{align*}
\mathcal{F}(k;x,p)=\sum_{n\ge0}x^{n}\sum_{\mu\le\mu_{0}}p^{\mathrm{Area}(\mu)},
\end{align*}
where $\mu$ is a $(1,k)$-Dyck path and $\mu_{0}=(UD^{k})^{n}$.
From Eq. (\ref{eq:GinArea}) in Proposition \ref{prop:GinArea}, we have $\mathcal{F}(k;x,p)=G^{(k)}(p^{-1}x,p)$.
\end{proof}

\section{Correspondence between two types of heaps for generalized Dyck paths}
\label{sec:ItoII}
Fix a positive integer $b\ge1$.
We have two types of heaps for a $(1,b)$-Dyck path studied in 
Sections \ref{sec:1b} and \ref{sec:GDGF}.
Recall that the first (resp. second) heaps are called type $I$ (resp. type $II$).
In this section, we construct a bijection between heaps of type $I$ and of type $II$.
Let $\mathcal{H}^{I}$ and $\mathcal{H}^{II}$ be heaps of type $I$ and type $II$ 
for a generalized $(1,b)$-Dyck path.
We first give the map $\kappa:\mathcal{H}^{I}\mapsto\mathcal{H}^{II}$, and give the 
inverse of $\kappa$.

Suppose that $\mathcal{H}^{I}$ consists of $r$ staircases $\mathcal{H}_{i}^{I}$,
$1\le i\le r$.
Recall that each staircase corresponds to a segment $[l_{i},r_{i}]$.
By the definition of heaps of type $I$, we always have 
$r_{i}=1+bp(i)$ for some positive integer $p(i)$.
Let $L_{i}=bp(i)-l_{i}$ and $R_{i}=r_{i}$.
Note that $p(i)$ takes value in $[1,n]$.
We put a piece $[L_{i}+1,L_{i}+b+1]$ at height $n+1-p(i)$ in $\mathcal{H}^{II}$ for 
all $i\in[1,r]$.
Recall that we need $n+1$ pieces of length $b$ in $\mathcal{H}^{II}$.
We need to fix other $n+1-r$ pieces in $\mathcal{H}^{II}$.
From the condition (H1), we put a piece $[1,b+1]$ at height $n+1$.
We have $r$ positive integers $p(i)\in[1,n]$, $1\le i\le r$. 
We extend this by adding $p(0)=0$.
The construction of a heap of type $I$ implies that all $r(i)$, $1\le i\le r$,
are distinct. 
In terms of $p(i)$, this means that $p(i)$, $0\le i\le r$, is strictly increasing.
For an integer $p(i)<q<p(i+1)$, we put 
a piece of length $b$ whose left abscissa is $L_{i}+b(q-p(i))+1$ at height $n+1-q$ in $\mathcal{H}^{II}$. 
This definition gives a unique heap of type $II$.

\begin{lemma}
A map $\kappa$ sends a heap $\mathcal{H}^{I}$ to a heap satisfying the conditions (H1) and (H2).
\end{lemma}
\begin{proof}
By construction of $\kappa$, it is obvious that 
a heap $\kappa(\mathcal{H}^{I})$ satisfies the condition (H1).
We show that the condition (H2) holds.
When $p(i)<q<p(i+1)$, the heap at height $n+1-q$-th cannot be a maximal piece
by construction.
Therefore, it is enough to show that a piece corresponding $p(i)$-th staircase
is not a maximal piece.
Recall that the left abscissa of a segment corresponding to the $i$-th 
staircase $\mathcal{H}_{i}^{I}$ is given by $bp(i)-l_{i}+1$. 
In a heap of type $I$, if two segments $[l_1,r_1]$ and $[l_2,r_2]$ satisfies 
$l_1<l_2$, then we have $r_1<r_2$. 
By combining these facts together, the left abscissa of a piece in $\mathcal{H}^{II}$ 
corresponding to $i$-th staircase $\mathcal{H}_{i}^{I}$ is left to the one of a piece 
at height $n+2-p(i)$.
This implies that a piece at height $n+1-p(i)$ is not a maximal piece. 
As a summary, the condition (H2) follows.
\end{proof}

\begin{example}
Consider a heap $\mathcal{H}^{I}$ of type $I$ with $(n,b)=(2,2)$, which 
consists of a single piece $[1,3]$ of length two.
Since $3=1+2\cdot 1$, we have $p(1)=1$. From this, we have 
a segment $[2,4]$ at height $2$.
Obviously, we have a segment $[1,3]$ at height $3$.
The segment at height $1$ is given by $[4,6]$.
In total, we have three pieces $[1,3]$, $[2,4]$ and $[4,6]$ 
from top to bottom. 
\end{example}

The inverse map $\kappa^{-1}:\mathcal{H}^{II}\mapsto\mathcal{H}^{I}$ is given 
as follows.
We first define a linear order on pieces in the heap $\mathcal{H}^{II}$.
Let $n+1$ be the number of pieces in the heap $\mathcal{H}^{II}$ of type $II$ 
and $p_{i}$, $1\le i\le n+1$, be ordered pieces such that $p_1>p_2>\ldots>p_{n+1}$.
We recursively define $p_{i}$ starting from $p_{1}$.
Since we have a unique maximal piece in a heap by construction (see the condition (H2)), 
we define the piece $p_{1}$ as this maximal piece.
We remove the piece $p_{1}$ from $\mathcal{H}^{II}$, and we define 
$p_{2}$ as the right-most maximal piece of $\mathcal{H}^{II}\setminus\{p_{1}\}$.
Similarly, $p_{3}$ is defined as the right-most maximal piece in 
$\mathcal{H}^{II}\setminus\{p_{1},p_{2}\}$.
We recursively define $p_{i}$ by removing the right-most maximal piece one-by-one
from $\mathcal{H}^{II}$.
Each piece $p_{i}$ corresponds to a segment $[l_i,r_i]$ such that 
$r_i-l_i=b$.
Let $\mathcal{I}(\mathcal{H}^{II})$ be the set of integers:
\begin{align*}
\mathcal{I}(\mathcal{H}^{II}):=
\{1\le i\le n+1 : r_{i-1}\neq l_{i} \},
\end{align*}
where we define $r_{0}:=1$.
When $r_{i-1}=l_{i}$, we have no peak in a heap of type $I$ at height $n+1-i$
since the length of a piece is $b$ and the difference of the left abscissae 
of pieces at height $n+2-i$ and $n+1-i$ in the heap of type $II$ is also $b$.
From these, it is clear that $|\mathcal{I}(\mathcal{H}^{II})|=r$ where 
$r$ is the number of staircases in a heap of type $I$.	
We define the set of segments $\mathcal{S}(\mathcal{H}^{II})$ 
constructed from the $n+1$ pieces $[l_i,r_i]$ by 
\begin{align*}
\mathcal{S}(\mathcal{H}^{II}):=
\{[L_i,R_i] | L_i= b(i-1)-l_{i}+1, R_i=1+b(i-1) , i\in\mathcal{I}(\mathcal{H}^{II}) \}.
\end{align*}
The set $\mathcal{S}(\mathcal{H}^{II})$ gives the segments for $\mathcal{H}^{I}$.
Each segment $[L_i,R_i]$ in $\mathcal{H}^{I}$ corresponds to a piece $p_i$ in $\mathcal{H}^{II}$.
The linear order of the segments $[L_i,R_i]$ is the reversed order of the 
pieces $p_{i}$.

Further, when $[L_i,R_i]\in\mathcal{S}(\mathcal{H}^{II})$, we have $L_{i}<R_i$ and 
$R_{i}=1+b(p-1)$ for some $p$.
This implies that a staircase ends at position $R_{i}$, which 
is compatible with the definition of heaps for $\mathcal{H}^{I}$.
It is easy to see that the value $L_{i}$ gives the distance of a peak 
from $x=bn$, and equivalently the left abscissa of the corresponding 
staircase in a heap of type $I$.

\begin{example}
We consider three heaps of type $II$ with $(n,k)=(2,2)$.
\begin{enumerate}
\item A heap consisting of three pieces $[1,3]$, $[2,4]$ and $[4,6]$ from top to bottom.
The set $\mathcal{I}$ is $\{2\}$ and we have a segment $[1,3]$ in a heap of type $I$.
\item
A heap consisting of $[1,3]$, $[3,5]$, $[4,6]$.
from top to bottom. Then, we have a segment $[1,5]$ in a heap of type $I$ 
since the set $\mathcal{I}$ is $\{3\}$.
\item 
A heap consisting of $[1,3]$, $[2,4]$ and $[3,5]$.
We have $\mathcal{I}=\{2,3\}$ and two segments $[1,3]$ and $[2,5]$.
By reversing the order of segments, we have a heap of type $I$ consisting 
of $[2,5]$ and $[1,3]$ from top to bottom.
\end{enumerate}
\end{example}

By constructions of $\kappa$ and $\kappa^{-1}$, we have a bijection between 
a staircase in a heap of type $I$ and a piece in a heap of type $II$.
It is obvious that we have $\kappa\circ\kappa^{-1}$ and $\kappa^{-1}\circ\kappa$
give the identity map on heaps.

As a summary, the next proposition follows from the explicit constructions of the map $\kappa$ and 
its inverse.
\begin{prop}
The map $\kappa$ is a bijection between $\mathcal{H}^{I}$ and $\mathcal{H}^{II}$.
\end{prop}

In Figure \ref{fig:bijGD}, we summarized the relation among heaps for 
generalized Dyck paths.
\begin{figure}[ht]
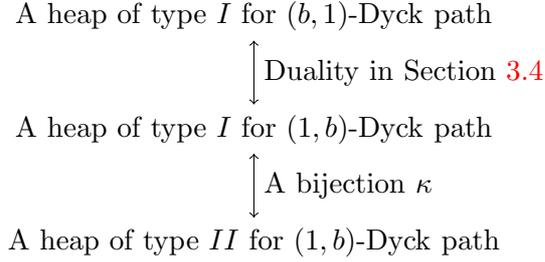

\tikzpic{-0.5}{
\node(1)at(0,1.5){A heap of type $I$ for $(b,1)$-Dyck path};
\node(0)at (0,0){A heap of type $I$ for $(1,b)$-Dyck path};
\node(-1)at(0,-1.5){A heap of type $II$ for $(1,b)$-Dyck path};
\draw[<->](1.south) to node[anchor=west]{Duality in Section \ref{sec:dual1bb1}}(0.north);
\draw[<->](0.south) to node[anchor=west]{A bijection $\kappa$}(-1.north);
}
\caption{Heaps for generalized Dyck paths.}
\label{fig:bijGD}
\end{figure}

\begin{example}
We have twelve heaps for $(1,k)$-Dyck paths of size $n$ with $(n,k)=(2,2)$.
The heaps of type $I$ in Figure \ref{fig:HPk2} respectively correspond to those of type $II$ in Figure \ref{fig:k2} by 
the bijection $\kappa$.
\end{example}

\section{Symmetric Dyck paths}
\label{sec:symDP}
In this section, we study symmetric Dyck paths and their relation to
heaps.
We first introduce a set of heaps with some conditions, and study 
the generating functions.
Then, we give two interpretations of these generating functions
in terms of statistics on symmetric Dyck paths and on Dyck paths respectively.

\subsection{Heaps of pieces with monomers and dimers}
The heaps for symmetric Dyck paths are defined on the set $P=\mathbb{N}_{\ge1}$ 
and pieces of $E$ are monomers and dimers.
We consider the set of heaps $H$ such that
\begin{enumerate}[(J1)]
\item A piece is a dimer or a monomer. In other words, the size of a piece
is one or zero. 
\item The abscissa of a monomer is one.
\item A heap $H$ has a unique maximal piece, and its left abscissa is one.
\end{enumerate}
The condition (J3) implies that a heap $H$ is a semi-pyramid since 
the maximal piece is placed at the left-most column.
Let $N(d)$ and $N(m)$ be the number of dimers and monomers in $H$.
We define the size of $H$ as $2N(d)+N(m)$.
The heaps satisfying the conditions (J1), (J2) and (J3) are a variant 
of heaps of type $II$ studied in Section \ref{sec:LPGF}.
The difference is that we consider two types of pieces, dimers and monomers.

In Figure \ref{fig:symn4}, we list up all heaps of size $4$ which 
satisfy the conditions (J1), (J2) and (J3).
Since the sizes of a monomer and a dimer is one and two respectively, 
the number of monomers in a heap is the same as the size modulo two.
Later, we will construct a bijection between heaps satisfying (J1) to (J3) 
and symmetric Dyck paths. 
\begin{figure}[ht]
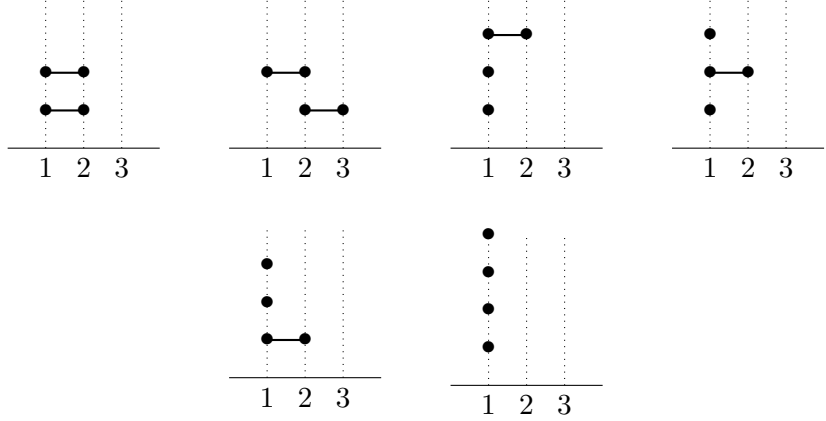

\tikzpic{-0.5}{[scale=0.5]
\draw(0,0)--(4,0);
\foreach \x in {1,2,3}{
\draw[dotted](\x,0)node[anchor=north]{$\x$}--(\x,4);
}
\draw[thick](1,2)node{$\bullet$}--(2,2)node{$\bullet$}(1,1)node{$\bullet$}--(2,1)node{$\bullet$};
}
\quad
\tikzpic{-0.5}{[scale=0.5]
\draw(0,0)--(4,0);
\foreach \x in {1,2,3}{
\draw[dotted](\x,0)node[anchor=north]{$\x$}--(\x,4);
}
\draw[thick](1,2)node{$\bullet$}--(2,2)node{$\bullet$}(2,1)node{$\bullet$}--(3,1)node{$\bullet$};
}
\quad
\tikzpic{-0.5}{[scale=0.5]
\draw(0,0)--(4,0);
\foreach \x in {1,2,3}{
\draw[dotted](\x,0)node[anchor=north]{$\x$}--(\x,4);
}
\draw[thick](1,3)node{$\bullet$}--(2,3)node{$\bullet$};
\draw (1,2)node{$\bullet$}(1,1)node{$\bullet$};
}
\quad
\tikzpic{-0.5}{[scale=0.5]
\draw(0,0)--(4,0);
\foreach \x in {1,2,3}{
\draw[dotted](\x,0)node[anchor=north]{$\x$}--(\x,4);
}
\draw[thick](1,2)node{$\bullet$}--(2,2)node{$\bullet$};
\draw(1,3)node{$\bullet$}(1,1)node{$\bullet$};
}
\\[11pt]
\tikzpic{-0.5}{[scale=0.5]
\draw(0,0)--(4,0);
\foreach \x in {1,2,3}{
\draw[dotted](\x,0)node[anchor=north]{$\x$}--(\x,4);
}
\draw[thick](1,1)node{$\bullet$}--(2,1)node{$\bullet$};
\draw(1,2)node{$\bullet$}(1,3)node{$\bullet$};
}
\quad
\tikzpic{-0.5}{[scale=0.5]
\draw(0,0)--(4,0);
\foreach \x in {1,2,3}{
\draw[dotted](\x,0)node[anchor=north]{$\x$}--(\x,4);
}
\foreach \x in {1,2,3,4}{
\draw(1,\x)node{$\bullet$};}
}

\caption{Six semi-pyramids for $n=4$}
\label{fig:symn4}
\end{figure}

We define the valuation of a heap $H$ by 
\begin{align*}
v(H):=x^{N(d)+N(m)}y^{2N(d)+N(m)}
p^{\sum(\text{left abscissae of pieces in $H$})}.
\end{align*}
The ordinary generating function for the heaps is defined as 
\begin{align*}
G(x,y,p):=\sum_{H}v(H).
\end{align*}

We apply Proposition \ref{prop:inversion} to the generating function $G(x,y,p)$.
For this purpose, we first compute the generating function $A(x,y,p)$ of trivial heaps 
without monomers.
Then, we compute the generating function $B(x,y,p)$ of trivial heaps with dimers 
and monomers by use of the generating function $A(x,y,p)$.
Let $A(x,y,p):=\sum_{n\ge0}(-1)^{n}a_{n}x^{n}$ be the ordinary generating function of trivial heaps
which contain only dimers.
The function $A(x,y,p)$ satisfies 
\begin{align*}
A(x,y,p)&=1+(-1)xy^2p A(p^2x,y,p)+(-1)xy^2p^2A(p^3x,y,p)+\cdots, \\
&=1+(-1)xy^2\sum_{j\ge1}p^{j}A(p^{j+1}x,y,p).
\end{align*}
Each term in the above equation corresponds to heaps such that there is a dimer at position $j$ and no
dimers left to it.
The factor $xy^2p^j$ corresponds to the dimer at position $j$, and $A(p^{j+1}x,y,p)$ corresponds to 
the generating function of heaps right to this dimer.
By solving the functional relation for the formal power series $A(x,y,p)$, 
the coefficient $a_{n}$ is explicitly given by 
\begin{align*}
a_{n}&=\genfrac{}{}{}{}{p^{2n-1}y}{1-p^n}a_{n-1}, \\
&=\genfrac{}{}{}{}{p^{n^2}y^{2n}}{(p;p)_{n}},
\end{align*}
where $n\ge 0$ and we have used the initial condition $a_{0}=1$.

Secondly, we compute the generating function $B(x,y,p)$ of trivial heaps with monomers 
and dimers.
Let $B(x,y,p):=\sum_{n\ge0}(-1)^{n}b_{n}x^{n}$ be the ordinary generating function of 
trivial heaps which contain dimers and monomers. 
Trivial heaps contributing to $B(x,y,p)$ are classified into two 
classes.
The first one is a trivial heap without a monomer, and 
the second is a trivial heap with a monomer.
Note that the position of a monomer is one if it exists.
From these, the function $B(x,y,p)$ satisfies 
\begin{align*}
B(x,y,p)=A(x,y,p)+(-1)xyp A(px,y,p).
\end{align*}
The second term is a contribution of trivial heaps with a monomer at position one.
The factor $xyp$ comes from the maximal monomer at position $1$.

By solving the functional equation above, the coefficients $b_{n}$ can be expressed 
in terms of the coefficient $a_{n}$ as follows: 
\begin{align*}
b_{n}=a_{n}+yp^{n}a_{n-1}.
\end{align*}
By a simple calculation with an explicit expression of $a_{n}$, we have 
\begin{align*}
b_{n}=(p^{n-1}y+1-p^{n})\genfrac{}{}{}{}{y^{2n-1}p^{n^2-n+1}}{(p;p)_{n}},
\end{align*}
where we have the initial condition $b_{0}=1$.

By applying Proposition \ref{prop:inversion} to our model, 
the generating function $G(x,y,p)$ is given by 
\begin{align*}
\displaystyle
G(x,y,p)&=\genfrac{}{}{}{}{A(px,y,p)}{B(x,y,p)}, \\
&=\genfrac{}{}{}{}{\displaystyle\sum_{n\ge0}(-1)^{n}x^{n}\genfrac{}{}{}{}{p^{n(n+1)}y^{2n}}{(p;p)_{n}}}
{\displaystyle\sum_{n\ge0}(-1)^{n}x^{n}\genfrac{}{}{}{}{(p^{n-1}y+1-p^{n})p^{n^2-n+1}y^{2n-1}}{(p;p)_{n}}}.
\end{align*}
To obtain a functional equation for $G(x,y,p)$, we rewrite the expression of $G(x,y,p)$
in terms of a continued fraction.

We define a continued fraction by 
\begin{align*}
[a_0,a_1,a_2,a_3,\ldots]:=
\cfrac{1}{1-\cfrac{a_0}{1-\cfrac{a_1}{1-\cfrac{a_2}{\cdots}}}}.
\end{align*}

We define a continued fraction $R(x,y,p)$ as 
\begin{align*}
R(x,y,p):=[a_{i}=p^{i+1}xy^2]^{-1}.
\end{align*}
Then, it is easy to show that $R(x,y,p)$ satisfies the following functional equation:
\begin{align}
\label{eq:recR}
R(x,y,p)=1-\genfrac{}{}{}{}{pxy^2}{R(xp,y,p)}.
\end{align}
Note that the formal series $R(x,y,p)^{-1}$ satisfies the $p$-shifted functional 
equation: $R(x,y,p)^{-1}=1+pxy^2R(x,y,p)^{-1}R(px,y,p)^{-1}$.
This functional equation reduced to the functional equation for the Catalan 
numbers if we set $(p,y)=(1,1)$.
As we will see later, this observation implies that we have two expressions of 
$G(x,y,p)$ in terms of statistics on Dyck paths and on symmetric Dyck paths respectively.

Let $G'(x,y,p)$ be a continued fraction of the following form:
\begin{align*}
G'(x,y,p):=(pxy)^{-1}([a_0=pxy,a_{i}=p^{i}xy^{2}]-1).
\end{align*}
Then, $G'(x,y,p)$ can be expressed in terms of $R(x,y,p)$ as 
\begin{align}
\label{eq:recG'}
G'(x,y,p)=\genfrac{}{}{}{}{1}{R(x,y,p)-pxy}.
\end{align}
From Eqs. (\ref{eq:recR}) and (\ref{eq:recG'}), the formal power series
$G'(x,y,p)$ satisfies the following functional relation:
\begin{align*}
G'(x,y,p)=\genfrac{}{}{}{}{1+p^2xyG'(px,y,p)}{1-pxy-pxy(-p+y+p^2xy)G'(px,y,p)}.
\end{align*}

\begin{prop}
\label{prop:GG'}
Let $G(x,y,p)$ and $G'(x,y,p)$ be the formal power series as above. 
Then, we have $G(x,y,p)=G'(x,y,p)$.
\end{prop}
\begin{proof}
We show that $G(x,y,p)$ satisfies the same functional relation (\ref{eq:recG'}) 
as $G'(x,y,p)$.
First, we rewrite the functional equation for $A(x,y,p)$ as follows. 
Note that  
\begin{align*}
A(xp,y,p)&=\sum_{n\ge0}(-1)^{n}x^n\genfrac{}{}{}{}{p^{n(n+1)}y^{2n}}{(p;p)_{n}}, \\
&=A(x,y,p)+\sum_{n\ge0}(-1)^{n}x^{n+1}\genfrac{}{}{}{}{p^{(n+1)^2}y^{2(n+1)}}{(p;p)_{n}}, \\
&=A(x,y,p)+pxy^2A(p^{2}x,y,p).	
\end{align*}
The fraction $\widetilde{A}(x,y,p):=A(px,y,p)A(x,y,p)^{-1}$ satisfies 
\begin{align*}
\widetilde{A}(x,y,p)=1+pxy^2\widetilde{A}(x,y,p)\widetilde{A}(px,y,p),
\end{align*}
which implies $\widetilde{A}(x,y,p)=R(x,y,p)^{-1}$ by Eq. (\ref{eq:recR}).
By definition of $G(x,y,p)$ in terms of $A(x,y,p)$ and $B(x,y,p)$, we have 
\begin{align*}
G(x,y,p)^{-1}&=\genfrac{}{}{}{}{A(x,y,p)-xypA(px,y,p)}{A(px,y,p)}, \\
&=\widetilde{A}(x,y,p)^{-1}-pxy, \\
&=R(x,y,p)-pxy.
\end{align*}
This expression coincides with the expression (\ref{eq:recG'}) and 
we have $G(x,y,p)=G'(x,y,p)$ as formal power series.
\end{proof}

\begin{remark}
The formal power series $\widetilde{A}(x,y,p)$ in the proof of Proposition \ref{prop:GG'} 
is the generating function of Dyck paths. 
In fact, $\widetilde{A}(x,y,p)$ is expressed as 
\begin{align*}
\widetilde{A}(x,y,p)
=
\sum_{n\ge0}(pxy^2)^{n}
\sum_{\mu\in\mathcal{D}_{n}}p^{\mathrm{Area}(\mu)},
\end{align*}
where $\mathcal{D}_{n}$ is the set of Dyck paths of size $n$ and 
the static $\mathrm{Area}(\mu)$ is the number of unit boxes 
below $\mu$ and above $(UD)^{n}$.
\end{remark}

\subsection{Generating function in terms of symmetric Dyck paths}
To express the generating function $G(x,y,p)$ in terms of 
symmetric Dyck paths, we introduce the notion of symmetric Dyck paths.
We first observe a bijection between a heap for $G(x,y,p)$ and a symmetric 
Dyck path. 
Then, we interpret the generating function $G(x,y,p)$ in terms of 
statistics on symmetric Dyck paths.

A {\it non-crossing perfect matching} of size $2n$ is a line graph with $2n$ points
such that each point is connected to another point by an arch which does not 
intersect with another arch.
The cardinality of the set of non-crossing perfect matchings of size $2n$ 
is given by the $n$-th Catalan number.
We have a natural bijetion between a non-crossing perfect matching 
and a Dyck path as follows.
Given an arch from $i$ to $j$, $i<j$, in a non-crossing perfect matching, 
we define a word $w=w_1\ldots w_{2n}$ of length $2n$ by $w_{i}=U$ and $w_{j}=D$.
The word $w$ represents a Dyck path.	

A {\it symmetric non-crossing perfect matching} of size $n$ is a graph with $n$ points 
obtained from a non-crossing perfect matching $p$ of size $2n$ by cutting $p$ into 
two pieces along the vertical line in the middle.
A symmetric non-crossing perfect matching consists of arches and half-arches.
	
In Figure \ref{fig:sDyck}, we list up all symmetric non-crossing perfect matchings, i.e., 
symmetric Dyck paths of size four. There are six symmetric non-crossing perfect matchings.
\begin{figure}[ht]
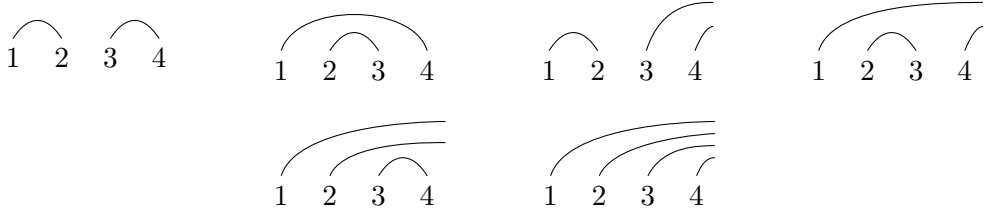

\tikzpic{-0.5}{[scale=0.4]
\draw(0,0)node[anchor=north]{$1$}..controls(0.5,0.8)and (1.1,0.8)..(1.6,0)node[anchor=north]{$2$};
\draw(3.2,0)node[anchor=north]{$3$}..controls(3.7,0.8)and (4.3,0.8)..(4.8,0)node[anchor=north]{$4$};
}\qquad
\tikzpic{-0.5}{[scale=0.4]
\draw(0,0)node[anchor=north]{$1$}..controls(0.5,1.6)and (4.3,1.6)..(4.8,0)node[anchor=north]{$4$};
\draw(1.6,0)node[anchor=north]{$2$}..controls(2.1,0.8)and (2.7,0.8)..(3.2,0)node[anchor=north]{$3$};
}\qquad
\tikzpic{-0.5}{[scale=0.4]
\draw(0,0)node[anchor=north]{$1$}..controls(0.5,0.8)and (1.1,0.8)..(1.6,0)node[anchor=north]{$2$};
\draw(3.2,0)node[anchor=north]{$3$}..controls(3.7,1.6)and(4.9,1.6)..(5.4,1.6)
(4.8,0)node[anchor=north]{$4$}..controls(5,0.5)and(5.2,0.8)..(5.4,0.8);
}\qquad
\tikzpic{-0.5}{[scale=0.4]
\draw(0,0)node[anchor=north]{$1$}..controls(0.5,1.6)and(4.3,1.6)..(5.4,1.6);
\draw(4.8,0)node[anchor=north]{$4$}..controls(5,0.5)and(5.2,0.8)..(5.4,0.8);
\draw(1.6,0)node[anchor=north]{$2$}..controls(2.1,0.8)and (2.7,0.8)..(3.2,0)node[anchor=north]{$3$};
}\\[11pt]
\tikzpic{-0.5}{[scale=0.4]
\draw(0,0)node[anchor=north]{$1$}..controls(0.5,1.6)and(4.3,1.8)..(5.4,1.8);
\draw(1.6,0)node[anchor=north]{$2$}..controls(2.1,1.2)and(5.2,1.1)..(5.4,1.1);
\draw(3.2,0)node[anchor=north]{$3$}..controls(3.7,0.8)and (4.3,0.8)..(4.8,0)node[anchor=north]{$4$};
}\qquad
\tikzpic{-0.5}{[scale=0.4]
\draw(0,0)node[anchor=north]{$1$}..controls(0.5,1.6)and(4.3,1.8)..(5.4,1.8);
\draw(1.6,0)node[anchor=north]{$2$}..controls(2.1,1.2)and(5.2,1.4)..(5.4,1.4);
\draw(3.2,0)node[anchor=north]{$3$}..controls(3.7,1)and(4.9,1)..(5.4,1)
(4.8,0)node[anchor=north]{$4$}..controls(5,0.5)and(5.2,0.6)..(5.4,0.6);
}
\caption{Symmetric non-crossing perfect matchings of size four}
\label{fig:sDyck}
\end{figure}
Each symmetric Dyck path in Figure \ref{fig:sDyck} corresponds to a heap 
in Figure \ref{fig:symn4}.
An arc in a symmetric Dyck path corresponds to a dimer, and a half-arc
corresponds to a monomer.

\begin{prop}
\label{prop:bijsym}
We have a bijection between the set of heaps $H$ of size $n$ satisfying the conditions (J1), (J2) and (J3) 
and the set of symmetric non-crossing perfect matchings $C$ of size $n$. 
\end{prop}
\begin{proof}

We construct a bijection between the two sets.
We first construct a map from a heap $H$ to a symmetric non-crossing perfect match $C$.
In a heap $H$, we have several monomers and several staircases which consist of dimers. 
Let $N^{d}_{i}$ (resp. $N^{m}_{i}$) be the number of dimers (resp. monomers) up to 
$i$-th piece from top in the heap $H$.
We define the integer $N_{i}:=2N^{d}_{i-1}+N^{m}_{i-1}+1$.
Then, the $i$-th monomer from top corresponds to a half-arch at the position $N_{i}$ in $C$.
Since the position of monomers is always one, monomers divide the heap into monomers and staircases 
consisting of only dimers. The staircase consisting of only dimers satisfies the properties 
of heaps of type $II$ studied in Section \ref{sec:LPGF}.
Especially, left abscissa of the maximal piece of such a staircase is one, and it is a semi-pyramid. 
Therefore, such staircases give (non-symmetric) non-crossing perfect matching.
On the other hand, in $C$, half-arches divide the symmetric non-crossing perfect matching 
into half-arches and (non-symmetric) non-crossing perfect matching.  
As a result, a monomer fixes the position of a half-arch in $C$, and staircases 
give a non-crossing perfect matching between half-arches.
This defines a map from $h$ to $C$.
 
Conversely, suppose we have a symmetric Dyck path $C$.
Then, first we decompose $C$ into arches and half-arches.
A half-arch corresponds to a monomer, and arches corresponds to a semi-pyramid such that 
the left abscissa of its maximal piece is one.
This semi-pyramid is a heap of type $II$.
This gives a map from a symmetric Dyck path to a heap.
 
From these observations, we have a bijection between the two sets, which completes 
the proof. 
\end{proof}

We define a weight to a symmetric non-crossing perfect matching $C$ as follows.
Given a arch $a$ connecting $i$ with $j$, $i<j$, in $C$, 
we define the weight $w(a)$ by one plus the number of arches connecting $i'$ with $j'$, 
$i'<j'$, such that $i'<i<j<j'$.
Similarly, the weight $w(b)$ given to a half-arch $b$ is defined to be one.
The weight $w(C)$ of $C$ is given by 
\begin{align*}
w(C):=\sum_{a\in C_{a}}w(a)+\sum_{b\in C_{h}}w(b),
\end{align*} 
where $C_a$ and $C_{h}$ are the set of arches and half-arches.

Let $\mathcal{C}_{n}$ be the set of symmetric non-crossing perfect matchings
and $|a(C)|$ be the number of arches in $C\in\mathcal{C}_{n}$, {\it i.e.}, 
the sum of the numbers of arches and half-arches.
\begin{prop}
\label{prop:GxypSym}
The generating function $G(x,y,p)$ can be expressed in terms of symmetric 
Dyck paths as 
\begin{align}
\label{eq:GxypSym}
G(x,y,p)=\sum_{n\ge0}y^{n}\sum_{C\in\mathcal{C}_{n}}x^{|a(C)|}p^{w(C)}.
\end{align}
\end{prop}
\begin{proof}
Let $H$ a heap of size $n$. From the condition (J2), a monomer contributes 
to the generating function as $xp$.
In terms of a half-arch $b$, we have $xp=xp^{w(b)}$. 
If the left abscissa of a dimer $d$ in $H$ is $n$, $d$ contributes 
to the generating function as $x^{2}yp^{n}$.
In terms of an arch $a$ in a symmetric non-crossing perfect matching, $n$ is equal to 
$w(a)$.
This follows from the correspondence between a heap and a symmetric Dyck path 
in Proposition \ref{prop:bijsym}.
From these observations, we have 
\begin{align*}
w(H)=y^{n}x^{|a(C)|}p^{w(C)},
\end{align*}
which completes the proof.
\end{proof}

\subsection{Generating function in terms of Dyck paths}
The generating function $G(x,y,p)$ has another expression in terms of 
(not necessarily symmetric) Dyck paths.
Recall that a Dyck path $\mu$ can be decomposed into a concatenation 
of prime Dyck paths, {\it i.e.}, $\mu=\mu_1\circ\ldots\mu_{\mathfrak{p}(\mu)}$
where the value $\mathfrak{p}(\mu)$ is the number of prime Dyck paths in $\mu$.
We denote by $\mathrm{Area}(\mu)$ the number of unit boxes
below $(UD)^{n}$ and above $\mu$ where $\mu$ is a Dyck path consisting of $n$ up 
and $n$ right steps.
\begin{prop}
\label{prop:GxypDyck}
The generating function $G(x,y,p)$ can be expressed in terms of 
Dyck paths as
\begin{align}
\label{eq:GxypDyck}
G(x,y,p)=
\sum_{n\ge0}(pxy)^{n}
\sum_{\mu\in\mathcal{D}_{n+1}}
p^{D(\mu)}y^{n+1-\mathfrak{p}(\mu)},
\end{align} 
where $\mathcal{D}_{n}$ is the set of Dyck paths of size $n$ and 
\begin{align}
\label{eq:Dmu}
D(\mu):=\mathfrak{p}(\mu)+\mathrm{Area}(\mu)-(n+1).
\end{align}
\end{prop}

To prove Proposition \ref{prop:GxypDyck}, we introduce a bijection 
between a symmetric Dyck path and a (non-symmetric) Dyck 
path.
Let $\mu$ be a symmetric Dyck path such that the sum of the number 
of arches and half-arches is $n$, and 
$\mu'$ be a Dyck path of size $n+1$.
We construct a correspondence between $\mu$ and $\mu'$.
We consider a word expression of $\mu$ in terms of an up step $u$ and a right step $d$.
The path $\mu'$ is obtained from $\mu$ as follows.
\begin{enumerate}
\item If a $u$ in $\mu$ is in a half-arch, we replace it with $du$.
\item If a $u$ and a $d$ in $\mu$ are connected by an arch, 
we leave them as they are. 
\item We append a $u$ and a $d$ to the sequence of $u$s and $d$s 
from left and right respectively.
We define $\mu'$ as the new sequence.
\end{enumerate}

\begin{example}
Consider a symmetric Dyck path $uudu$. The first and third $u$s are in half-arches,
and the second $u$ and a unique $d$ are connected by an arch.
\begin{align*}
\mu=uudu\xrightarrow{(1),(2)} DUudDU \xrightarrow{(3	)} UDUUDDUD,
\end{align*}
where $U$ and $D$ are new letters replaced from $u$ and $d$.
Therefore, we have $\mu'=uduuddud$.

Similarly, a symmetric Dyck path $\mu=uduud$ is mapped to $\mu'=uudduudd$.
\end{example}

The inverse map from $\mu'$ to $\mu$ is as follows.
\begin{enumerate}
\item We delete $u$ and $d$ in $\mu'$ from left and right respectively.
\item If the first letter is $d$, then the second letter is $u$. 
We replace this pair of $u$ and $d$ with a single letter $u$.
\item If the first letter is $u$, take a maximal prime Dyck path starting 
from this $u$. Then, we leave the letters in this prime Dyck path as they are.
\item We apply (2) and (3) to the remaining sequence until we replace all the letters with 
new letters.
\end{enumerate}
Since we delete the first $u$ from $\mu'$ in the process (1), the appearance of $du$ 
in the process (2) is obvious if it exists. In other words, we have no sequence 
$dd$ in (2) and (3).

For example, if $\mu'=uduududd$, we have 
\begin{align*}
\mu'=uduududd\rightarrow duudud \rightarrow Uudud\rightarrow UUDud\rightarrow UUDUD=\mu,
\end{align*}
where $U$ and $D$ are new letters.

As a summary, we have the following lemma.
\begin{lemma}
\label{lemma:bijSymDD}
The above map from $\mu$ to $\mu'$ is a bijection.
\end{lemma}

The following corollary is a direct consequence of Lemma \ref{lemma:bijSymDD}.
\begin{cor}
\label{cor:DycksymDyck}
The number of Dyck paths of size $n+1$ is equal to the number of 
symmetric Dyck paths of size $n$ where $n$ is the sum of the numbers of 
arches and half-arches.
\end{cor}

\begin{example}
Set $n=2$. We show the correspondence between Dyck paths and symmetric Dyck paths in Table \ref{table:DsD}.
\begin{table}[h]
\begin{tabular}{c|c|c|c|c|c}  
Dyck path & $UUUDDD$ & $UUDUDD$ & $UDUUDD$ & $UUDDUD$ & $UDUDUD$ \\	\hline
Symmetric Dyck path & $UUDD$ & $UDUD$ & $UUD$ & $UDU$ & $UU$ \\
\end{tabular}
\vspace{12pt}
\caption{A correspondence between Dyck paths and symmetric Dyck paths}
\label{table:DsD}
\end{table}
In Table \ref{table:DsD}, the left two symmetric Dyck paths have no half-arches,
the next two have a single half-arch, and the right-most one has two half-arches.
\end{example}

\begin{proof}[Proof of Proposition \ref{prop:GxypDyck}]
We compare the right hand side of Eq. (\ref{eq:GxypSym}) 
with the right hand side of Eq. (\ref{eq:GxypDyck}).
We have constructed a bijection between a symmetric Dyck path $C$ with $|a(C)|$ arches 
and a Dyck path of size $|a(C)|+1$ in Lemma \ref{lemma:bijSymDD}.
To show Eq. (\ref{eq:GxypDyck}), it is enough to show that 
the valuation of a symmetric Dyck path $\mu$ coincides with that of a Dyck path $\mu'$
under this bijection.
Suppose that we have $n_{a}$ arches and $n_{ha}$ half-arches in a symmetric 
Dyck path $C$.
From Eq. (\ref{eq:GxypSym}), the exponent of $y$ is $2n_a+n_{ha}$ for a symmetric 
Dyck path. 
On the other hand, the exponent of $y$ for a Dyck path is $y^{n_a+n_{ha}}$ from 
the factor $(pxy)^{n}$, and $y^{n_a}$ from $(n+1)-\mathfrak{p}(\mu')=n_a$.
We have $(n+1)-\mathfrak{p}(\mu')=n_a$ since we replace $u$ in $\mu$ by $du$ in $\mu'$ and 
this replacement reflects the number of prime Dyck paths in $\mu'$ and the number of 
half-arches in $\mu$.
Therefore, the exponents of $y$ coincide with each other.
Similarly, it is easy to see that the exponents of $x$ are $n_a+n_{ha}$ for the two paths 
under the bijection.

The exponent of $p$ for a symmetric Dyck path $\mu$ is given by $w(C)$ where $C$ is 
a symmetric non-crossing perfect matching for $\mu$.
Suppose that $\mu$ contains a maximal Dyck path $\mu_{1}$ consisting of arches.
The weight given to this Dyck path $\mu_1$ is the number of unit boxes 
below $\mu_{1}$ and above $(DU)^{m}$ where $m$ is the size of $\mu_{1}$.
Recall that the contribution of a half-arch in $\mu$ to $w(C)$ is one.
The exponent of $p$ for $\mu$ counts the number of boxes below a path $\mu$ and above the path 
$(DU)^{t}$ where $\mu'$ is obtained from $\mu$ by replacing an up step $u$ on a half-arch 
by $du$ and $t:=n_{a}+n_{ha}$. 
On the other hand, we have the exponent $n_{a}+n_{ha}$ of $p$ from the factor $(xyp)^{n}$ 
for $\mu'$.
Further, we have the exponent $D(\mu')$ from Eq. (\ref{eq:GxypDyck}).
By comparing the the exponent for $\mu$ with that for the Dyck path $\mu'$, one can show that 
two exponents coincide with each other by choosing $D(\mu')$ as Eq. (\ref{eq:Dmu}).

From these observations, since we have shown that the bijection preserves the weights, 
we have Eq. (\ref{eq:GxypDyck}).
\end{proof}

\begin{example}
Let $\mu=UUDUUDDU$ and $\mu'=UDUUDUUDDDUD$ be a symmetric Dyck path and a Dyck path 
respectively. 
We have $n=5$, and $\mathfrak{p}(\mu')=3$, $\mathrm{Area}(\mu')=4$.
Therefore, the both paths have the same valuation $p^6x^5y^{8}$.
\end{example}

\begin{remark}
We compare the expression in Proposition \ref{prop:GxypSym} and the one in 
Proposition \ref{prop:GxypDyck}.
As a formal power series, two expressions are the same. 
The difference between the two expressions is that whether we regard $G(x,y,p)$ 
as a formal power series in terms of $y$ or $x$.
Note that the expression in Proposition \ref{prop:GxypDyck} involves not only 
symmetric Dyck paths but also non-symmetric Dyck paths due to the bijection 
between Dyck paths and symmetric Dyck paths given in Corollary \ref{cor:DycksymDyck}.
\end{remark}

\bibliographystyle{amsplainhyper} 
\bibliography{biblio}

\end{document}